\documentclass[11pt]{amsart}
\usepackage{mathptmx}
\usepackage{amsmath}
\usepackage{amsbsy}  
\usepackage{amsfonts}
\usepackage{amsthm}
\usepackage{amssymb} 
\usepackage{array}
\usepackage{mathbbol}
\usepackage{mathtools}
\usepackage{graphicx}
\usepackage{graphics}
\usepackage{tkz-euclide}
\usepackage{bigints}
\usepackage{empheq}
\usepackage{float}
\usepackage{a4wide}
\usepackage[labelformat=simple]{subcaption}
\usepackage[bookmarks=true,colorlinks=true, pdfstartview=FitV, linkcolor=black, 
citecolor=blue, urlcolor=black]{hyperref}

\newcommand{\pFq}[5]{{}_{#1}F_{#2} \left( {#3 \atop #4 }; #5 \right) }

\newcommand{\dP}[3]{\left.\frac{\mathrm{d}^{#1}{#3}}{\mathrm{d}{#2}^{#1}}\right|_{{#2}=-1}}
\newcommand{\dPa}[3]{\left.\frac{\mathrm{d}^{#1}{#3}}{\mathrm{d}{#2}^{#1}}\right|_{{#2}=-a}}

\newcommand{\dL}[2]{\left.\frac{\mathrm{d}^{#1}{L_{#2}^{(\alpha)}}}{\mathrm{d} x^{#1}}\right|_{x=0} }

\numberwithin{equation}{section}

\newtheorem{theorem}{Theorem}[section]
\newtheorem{proposition}[theorem]{Proposition}
\newtheorem{lemma}[theorem]{Lemma}
\newtheorem{corollary}[theorem]{Corollary}
\newtheorem{Remark}[theorem]{Remark}
\newenvironment{remark}{\begin{Remark}\rm}{\end{Remark}}

\newcommand{\ds}{\displaystyle}

\newcommand{\md}{\mathrm{d}}
\newcommand{\dt}{\mathrm{d}t}

\newcommand{\ixx}{\int_{-1}^{x+1}}
 
\begin{document}

\title{Volterra-Type Convolution of Classical Polynomials}

\author[Ana Loureiro]{Ana F.  Loureiro}
\address{{\normalfont Ana F Loureiro}: School of Mathematics, Statistics and Actuarial Sciences, University of Kent, Sibson Building, CT2 7FS,  U.K.}
\email{A.Loureiro@kent.ac.uk}
\thanks{}

\author{Kuan Xu}
\address{{\normalfont Kuan Xu}: School of Mathematics, Statistics and Actuarial Sciences, University of Kent, Sibson Building, CT2 7FS,  U.K.}
\email{K.Xu@kent.ac.uk}




\begin{abstract}
We present a general framework for calculating the Volterra-type convolution of 
polynomials from an arbitrary polynomial sequence $\{P_k(x)\}_{k \geqslant 
0}$ with $\deg P_k(x) = k$. Based on this framework, series representations 
for the convolutions of classical orthogonal polynomials, including Jacobi and 
Laguerre families, are derived, along with some relevant results pertaining to 
these new formulas. 
\end{abstract}

\maketitle

\section{Introduction}
Convolution, as a fundamental operation, is commonly seen in many fields of 
sciences and engineering. Let $f:[a,b]\mapsto \mathbb{C}$ and $g:[c,d]\mapsto 
\mathbb{C}$ be two continuous integrable functions defined on intervals with a 
same length, that is $b-a = d-c$, where $a$ and $c$ are finite numbers while 
$b$ and $d$ can be finite or infinite. Their convolution $h(x)$ is a third 
function, given by
\begin{equation}
h(x) = (f\ast g)(x) = \int_c^{x-a} f(x-t) g(t)\dt,~~~~~~~~~~~~x \in [a+c, a+d], 
\label{conv0}
\end{equation}
where the domain of $h(x)$, i.e. $[a+c,a+d]$, has the same length as those of 
$f(x)$ and $g(x)$. This operation is often denoted by an asterisk, as in 
\eqref{conv0}.

When $b$ and $d$ are finite, $f(x)$ and $g(x)$ are compactly supported and can 
be mapped to the interval $[-1, 1]$ via changes of variables and the convolution 
of the mapped versions of $f(x)$ and $g(x)$ differs from that of the original 
$f(x)$ and $g(x)$ by an affine transform only. Therefore, with slight abuse of 
our notation, we consider exclusively the convolution of two functions $f(x)$ 
and $g(x)$ that are defined on $[-1, 1]$, that is,
\begin{equation}
h(x) = (f\ast g)(x) = \int_{-1}^{x+1} f(x-t) g(t)\dt,~~~~~~~~~~~~x \in [-2, 0], 
\label{conv1}
\end{equation}
where the convolution $h(x)$ is, in this case, a function on $[-2, 0]$. 
Analogously, when $b$ and $d$ are infinities \eqref{conv0} becomes
\begin{equation}
h(x) = (f\ast g)(x) = \int_{0}^{x} f(x-t) g(t)\dt,~~~~~~~~~~~~x \in [0, \infty], 
\label{conv2}
\end{equation}
up to a real M\"{o}bius transform. Note that the domains of the transformed 
$f(x)$, $g(x)$, and the convolution $h(x)$ all become $[0, \infty]$ in
\eqref{conv2}.

A powerful working paradigm that motivates this investigation and are commonly 
adopted in problems where functions considered are smooth is to replace $f(x)$ 
and $g(x)$ by their unique series representation in terms of classical 
orthogonal polynomials, e.g. Chebyshev or (weighted) Laguerre series for $f(x)$ 
and $g(x)$ in \eqref{conv1} and \eqref{conv2} respectively. In numerical 
computation, such series are usually truncated at certain degrees so that the 
finite series accurate to machine precision can serve as good approximants. In 
either case, the calculation of convolution integrals boils down to the 
convolution of polynomial series of finite or infinite degrees, or, further, to 
the convolution of classical orthogonal polynomials.

Suppose that $f_M(x) = \sum_{m=0}^M a_m P_m(x)$ and $g_N(x) = \sum_{n=0}^N b_n 
P_n(x)$ are two finite series of classical orthogonal polynomials. The 
convolution of $f_M(x)$ and $g_N(x)$ results in a degree $M+N+1$ polynomial 
$h_{M+N+1}(x) = \sum_{k=0}^{M+N+1}c_k P_k(x)$. In a recent work \cite{XL}, 
it is shown that the convolution operator 
\begin{equation*}
V[f_M](g_N) = h_{M+N+1}(x) = \ixx f_M(x-t) g_N(t)\dt,~~~~~x \in [-2, 0],
\end{equation*}
which is defined by $f_M(x)$ and applied to $g_N(x)$ can be represented as a 
$(M+N+2) \times (N+1)$ matrix $R$ so that the coefficients vector  
$\underline{c}=(c_1, c_2, \ldots, c_{M+N+1})^T$ equals the product of $R$ and 
the column vector $\underline{b}=(b_1, b_2, \ldots, b_N)^T$. That is,
\begin{equation}\label{cRb}
\underline{c} = R\underline{b}.
\end{equation}
In fact, this convolution matrix $R$ is constructed numerically using a four- 
or five-point recurrence relation satisfied by its entries. However, neither 
coefficients $\{c_k\}_{k=0}^{M+N+1}$ nor the entries of the convolution matrix 
$R$ are known explicitly. 

In this paper, we derive explicit formulas for the convolution of two classical
orthogonal polynomials of Jacobi or Laguerre families. The new results may hint 
us on the rich structure of convolution matrices $R$ and, in turn, shed light 
upon their fast construction as well as the design of fast algorithms for 
convolving polynomial series. On a different note, the results presented in this 
paper largely expand the collection of the existing convolution formulas 
comprised of those of Laguerre polynomials $L^{(0)}_n(x)$ \cite[Eq. 
(18.17.2)]{DLMF} and Bessel functions of the first kind $J_n(x)$ \cite[Eq. 
(10.22.31)]{DLMF}.

In the next section, we present a framework for calculating the convolution of 
two general polynomials, along with some relevant results. Based on these, 
convolution formulas for polynomials of the Jacobi family are derived in 
Section 3, where we also detail the special cases of Gegenbauer, Legendre, and 
Chebyshev of the second kind. In Section 4, we show the convolution formulas for 
the Laguerre polynomials, which are also direct results of Section 2, before we 
conclude the paper by a brief discussion regarding the convolution of polynomial 
series in Section 5.

\section{Convolution of two elements in a polynomial sequence}

Let $\{P_n(x)\}_{n\geqslant 0}$ be a polynomial sequence with $\deg P_n(x) = 
n$, which forms a basis of the vector space of polynomials with complex 
coefficients. We consider the problem of finding explicit expressions for the 
$P_n(x)$-series coefficients $\rho_{j,n}^m$ so that
\begin{equation}\label{gen conv}
\int_{-a}^{x+a} P_m(x-t)  P_n(t) \mathrm{d}t = \sum_{j=0}^{m+n+1} 
\rho_{j,n}^{m} P_{j}(x+a),
\end{equation}
where $a$ is constant. When $a = 1$ and $a=0$, \eqref{gen conv} corresponds to 
the standard cases \eqref{conv1} and \eqref{conv2}, respectively. A change of 
variable $t\to x-\tau$ shows the commutativity of the convolution in \eqref{gen 
conv}, which gives a remarkable symmetry property
\begin{equation*}\label{gen conv sym}
\rho_{j,n}^m = \rho_{j,m}^n,
\end{equation*} 
for any $j, n, m \geqslant 0$. Therefore, there is no loss of generality if one 
assumes $n\geqslant m$ or $m \geqslant n$.

A sequence $\{\frac{\md^s}{\md x^s} P_{n+s}(x)\}_{s \geqslant 0}$ for $n 
\geqslant 0$, i.e. the $s$-th derivative of the original sequence, also spans 
the vector space of polynomials. Therefore, the $r$-th derivatives of $P_n(x)$ 
can be represented by a linear combination of the elements of 
$\{\frac{\md^s}{\md x^s} P_{n+s}(x)\}_{s \geqslant 0}$ as
\begin{equation}\label{cc genP}
\frac{\md^r P_n(x)}{\md x^r} = \sum_{k=0}^{n-r} \gamma_{n-r,k}^{(r, s)} 
\frac{\md^s P_{k+s}(x)}{\md x^s}, \quad n \geqslant r,
\end{equation}
where the coefficients $\gamma_{n-r,k}^{(r, s)}$ are referred to as the {\it 
connection coefficients} between $\{\frac{\md^r}{\md x^r} P_n(x)\}_{n \geqslant 
r}$ and $\{\frac{\md^s}{\md x^s} P_{n+s}(x)\}_{n \geqslant 0}$. 
The unique representation of $P_n(x)$ in terms of the 
monomial sequence $\{(x+a)^n\}_{n\geqslant 0}$ can be obtained by its Taylor 
expansion about $x=-a$
\begin{equation*}\label{Pn to xn}
P_n(x) = \sum_{k=0}^{n} \frac{1}{k!} \dPa{k}{x}{P_n(x)} (x+a)^k, \quad n 
\geqslant 0.
\end{equation*}
and, reversely, a unique set of coefficients $b_{n,k}$ exist such that 
\begin{equation}\label{xn to Pn gen}
(x+a)^n = \sum_{k=0}^{n} b_{n,k} P_k(x), \quad n \geqslant 0. 
\end{equation}
When $\{P_n(x)\}_{n \geqslant 0}$ is an orthogonal sequence, these 
$b$-coefficients can be obtained via the orthogonality measures and their 
moments. 

\begin{lemma}\label{lemma:rho}
The $\gamma$-connection coefficients in \eqref{cc genP} can be expressed in 
terms of the connection $b$-coefficients in \eqref{xn to Pn gen} by 
\begin{equation}\label{gamma with b}
\gamma_{n-r,k}^{(r,s)} = \sum_{\sigma=0}^{n-(r+k)} 
\frac{b_{\sigma+k+s,k+s}}{(\sigma+k+s)!} \dPa{r+k+\sigma}{x}{P_n(x)}.
\end{equation}
\end{lemma}
\begin{proof} The Taylor expansion of $\frac{\md^r}{\md x^r} P_n(x)$ gives 
\[
\frac{\md^r P_n(x)}{\md x^r} = \sum_{\sigma=0}^{n-r}\frac{1}{\sigma!} 
\dPa{r+\sigma}{x}{P_n(x)}(x+a)^\sigma = 
\sum_{\sigma=0}^{n-r}\frac{1}{(\sigma+s)!} 
\dPa{r+\sigma}{x}{P_n(x)} \left(\frac{\md^s}{\md x^s}(x+a)^{\sigma+s}\right) 
\]
where $s$ is an arbitrary positive integer. Substituting \eqref{xn to Pn gen} 
into the last equation gives
\[
\frac{\md^r P_n(x)}{\md x^r} = \sum_{\sigma=0}^{n-r}\frac{1}{(\sigma+s)!} 
\dPa{r+\sigma}{x}{P_n(x)} \left( \sum_{k=0}^{\sigma} b_{\sigma+s,k+s} 
\frac{\md^s}{\md x^s}P_{k+s}(x) \right)
\]
which, after exchanging the summations, becomes 
\[
\frac{\md^r P_n(x)}{\md x^r} = \sum_{k=0}^{n-r} \left( \sum_{\sigma=0}^{n-(r+k)}
\frac{b_{\sigma+k+s, k+s}}{(\sigma+k+s)!} \dPa{r+k+\sigma}{x}{P_n(x)}   
\right)\frac{\md^s}{\md x^s}P_{k+s}(x). 
\]
Matching the like terms in \eqref{cc genP} leads to \eqref{gamma with b}. 
\end{proof}

We omit the proof of the following lemma, which is concerned with the $p$-th 
derivative of the convolution in \eqref{gen conv} and can be easily shown by 
repeatedly applying the Leibniz rule for differentiation under the integral 
sign.
\begin{lemma}\label{lemma:pth}
For a positive integer $p$, 
\begin{equation} \label{pth of int P}
\frac{\mathrm{d}^p}{\mathrm{d}x^p} \int_{-a}^{x+a} P_m(x-t) P_n(t) \mathrm{d}t
= \int_{-a}^{x+a} \frac{\mathrm{d}^p}{\mathrm{d}x^p} P_m(x-t) P_n(t) \mathrm{d}t
+ \sum_{k=1}^p \dPa{p-k}{x}{P_m(x)} \frac{\mathrm{d}^{k-1}}{\mathrm{d}x^{k-1}} 
P_n(x+a).
\end{equation}
\end{lemma}

With Lemmas \ref{lemma:rho} and \ref{lemma:pth}, we show in the following 
theorem that $\rho_{j,n}^{m}$ in \eqref{gen conv} can be represented in terms 
of the $\gamma$-coefficients in \eqref{cc genP} and the $b$-coefficients in 
\eqref{xn to Pn gen}. 

\begin{theorem}\label{thm: rho k geq m} 
For $0\leqslant j\leqslant m+n+1$, the coefficients $\rho_{j,n}^m$ in \eqref{gen 
conv} can be expressed as
\begin{equation}\label{rho jn v1}
\rho_{j,n}^m = \sum_{p=j}^{m+n+1} \frac{b_{p,j}}{p!} \sum_{\nu=1}^p\left( 
\dPa{p-\nu}{x}{P_m(x)} \dPa{\nu-1}{x}{P_n(x)} \right),
\end{equation}
or, equivalently,
\begin{subequations}\label{rho jn v23}
\begin{equation}\label{rho jn v2}
\rho_{j,n}^m = \sum_{\nu=\max(1,j-n)}^{m+1}\left( 
\gamma_{n-j+\nu,0}^{(j-\nu,j)} 
\dPa{\nu-1}{x}{P_m(x)} \right) ~~~~~\text{ for } 
j\geqslant m+1
\end{equation}
and
\begin{equation}\label{rho jn v3}
\begin{multlined}
\rho_{j,n}^m = \sum\limits_{\nu=1}^{j}\left( \gamma_{m-j+\nu,0}^{(j-\nu,j)} 
\dPa{\nu-1}{x}{P_n(x)}\right)\\
\hspace{2cm}+ \sum\limits_{\nu=j+1}^{n+1}\left(\dPa{\nu-1}{x}{P_n(x)} \sum 
\limits_{p=0}^{m} \frac{b_{p+\nu,j}}{(p+\nu)!}\dPa{p}{x}{P_m(x)} \right)
~~~~~\text{ for } 0 \leqslant j\leqslant m, 
\end{multlined}
\end{equation}
\end{subequations}
where the $\gamma$- and the $b$-coefficients are the connection coefficients 
given in \eqref{cc genP} and \eqref{xn to Pn gen}, respectively. 
\end{theorem}

\begin{proof} 
To show \eqref{rho jn v1}, we Taylor expand the convolution integral in 
\eqref{gen conv} about $x=-2a$ 
\begin{equation}\label{Taylor of gen int}
\int_{-a}^{x+a} P_m (x-t) P_n(t) \mathrm{d}t = \sum_{p=0}^{m+n+1} 
\frac{1}{p!}\left[ \frac{\md ^p}{\md (x+a)^p}\int_{-a}^{x+a} P_m(x-t) P_n(t) 
\mathrm{d}t \right]_{x=-2a} (x+2a)^p.
\end{equation}
Note that the Taylor coefficients in \eqref{Taylor of gen int} can be obtained 
using \eqref{pth of int P}:
\[
\left[ \frac{\md ^p}{\md (x+a)^p}\int_{-a}^{x+a} P_m(x-t) P_n(t) \mathrm{d}t 
\right]_{x=-2a} = \sum_{\nu=1}^p \dPa{p-\nu}{x}{P_m(x)} \dPa{\nu-1}{x}{P_n(x)},
\]
where the sum is assumed zero when $p=0$, and $(x+2a)^p$ can be replaced by its 
expansion in \linebreak {$\{P_n(x+a)\}_{n\geqslant 0}$} as given in \eqref{xn to Pn gen}:
\[
(x+2a)^p = \sum_{j=0}^p b_{p,j} P_j(x+a).
\]
We substitute the last two equations into \eqref{Taylor of gen int} and exchange 
the order of the summations to have
\begin{equation}\label{int expansion any}
\int_{-a}^{x+a} P_m (x-t) P_n(t) \mathrm{d}t =\sum_{j=0}^{m+n+1} \left( 
\sum_{p=j}^{m+n+1} \frac{b_{p,j}}{p!} \sum_{\nu=1}^p \dPa{p-\nu}{x}{P_m(x)} \ 
\hspace{-3mm}\dPa{\nu-1}{x}{P_n(x)} \right) P_j(x+a).
\end{equation}
Matching terms in \eqref{int expansion any} and \eqref{gen conv} gives 
\eqref{rho jn v1}. 

To see \eqref{rho jn v2}, we take the $j$-th derivative on both sides of 
\eqref{gen conv} for $j \geqslant m+1$ to have 
\begin{equation*}\label{proof pth of int}
\frac{\mathrm{d}^j}{\mathrm{d}x^j} \int_{-a}^{x+a} P_m(x-t) P_n(t) \mathrm{d}t
=\sum_{k=0}^{m+n+1-j} \rho_{k+j,n}^m \frac{\mathrm{d}^j}{\mathrm{d}x^j} 
P_{k+j}(x+a).
\end{equation*}
Meanwhile, Lemma \ref{lemma:pth} gives
\begin{equation*}\label{interm0}
\frac{\mathrm{d}^j}{\mathrm{d}x^j} \int_{-a}^{x+a} P_m(x-t) P_n(t) \mathrm{d}t 
= \sum_{\ell=1}^j \dPa{j-\ell}{x}{P_m(x)} 
\frac{\mathrm{d}^{\ell-1}}{\mathrm{d}x^{\ell-1}} P_n(x+a),
\end{equation*}
where the convolution integral on the right-hand side of \eqref{pth of int 
P} vanishes here, since the integrand becomes zero for $j \geqslant m+1$. 
Combining the last two equations, we have
\begin{equation}\label{interm1}
\sum_{k=0}^{m+n+1-j} \rho_{k+j,n}^m \frac{\mathrm{d}^j}{\mathrm{d}x^j} 
P_{k+j}(x+a) = \sum_{\ell=1}^j \dPa{j-\ell}{x}{P_m(x)} 
\frac{\mathrm{d}^{\ell-1}}{\mathrm{d}x^{\ell-1}} P_n(x+a).
\end{equation}
If we denote by $S$ the sum on the right-hand side of \eqref{interm1},
\begin{equation*}
S = \sum_{\ell=0}^{j-1} \dPa{j-\ell-1}{x}{P_m(x)} \frac{\mathrm{d}^{\ell}} 
{\mathrm{d} x^{\ell}} P_n(x+a) = \sum_{\ell=j-(m+1)}^{n}
\dPa{j-\ell-1}{x}{P_m(x)} \frac{\mathrm{d}^{\ell}}{\mathrm{d} x^{\ell}}P_n(x+a),
\end{equation*}
where the last equality is obtained by noting that the summand disappear when 
$j-\ell-1 > m$ and $\ell > n$. Now we use the connection formula \eqref{cc genP} 
once again to have
\begin{equation}\label{interm2}
\begin{multlined}
S = \sum_{\ell=j-(m+1)}^n \dPa{j-\ell-1}{x}{P_m(x)} \sum_{k=0}^{n-\ell} 
{\gamma}_{n-\ell,k}^{(\ell, j)} \frac{\mathrm{d}^j}{\mathrm{d}x^j}P_{k+j}(x+a)\\ 
\hspace{3cm}= \sum_{k=0}^{m+n+1-j} \left( {\gamma}_{n-\ell,k}^{(\ell, 
j)} \sum_{\ell=j-(m+1)}^{n-k} \dPa{j-\ell-1}{x}{P_m(x)} \right)
\frac{\mathrm{d}^j}{\mathrm{d}x^j}P_{k+j}(x+a),
\end{multlined}
\end{equation}
where we have swapped the order of the sums. Combining \eqref{interm1} and
\eqref{interm2} and matching terms yield
\begin{equation}\label{interm3}
\rho_{k+j,n}^m = \sum_{\ell=j-(m+1)}^{n-k} {\gamma}_{n-\ell,k}^{(\ell, j)} 
\dPa{j-\ell-1}{x}{P_m(x)} 
\end{equation}
for $0 \leqslant k \leqslant m+n+1-j$. Particularly, for $k=0$, \eqref{interm3} 
becomes
\begin{equation*}\label{rho big k f1} 
\rho_{j,n}^m = \sum_{\ell=j-(m+1)}^n {\gamma}_{n-\ell,0}^{(\ell, j)}
\dPa{j-\ell-1}{x}{P_m(x)} = \sum_{\nu=j-n}^{m+1} \gamma_{n-j+\nu,0}^{(j-\nu,j)} 
\dPa{\nu-1}{x}{P_m(x)},
\end{equation*}
where we use the change of variable $\ell = j-\nu$ in the last step. Ensuring 
$\nu-1 \geqslant 0$, we obtain \eqref{rho jn v2}.

To see \eqref{rho jn v3}, we swap the order of sums in \eqref{rho jn v1} to 
have
\[
\rho_{j,n}^m = \left( \sum_{\nu=1}^{j} \sum_{p=j}^{m+\nu} + 
\sum_{\nu=j+1}^{n+1} 
\sum_{p=\nu}^{m+\nu} \right) \left(\frac{b_{p,j}}{p!} \dPa{p-\nu}{x}{P_m(x)}  
\dPa{\nu-1}{x}{P_n(x)}\right),
\]
since $\dPa{p-\nu}{x}{P_m(x)}$ and $\dPa{\nu-1}{x}{P_n(x)}$ vanish for $p > 
m+\nu$ and $\nu>n+1$, respectively. This is equivalent to 
\begin{equation}\label{rho jn v3 intermediate}
\begin{multlined}
\rho_{j,n}^m = \sum_{\nu=1}^{j} \dPa{\nu-1}{x}{P_n(x)} \sum_{p=0}^{m+\nu-j} 
\left(\frac{b_{p+j,j}}{(p+j)!} \dPa{p-\nu+j}{x}{P_m(x)}\right) \\ 
\hspace{4cm}+ \sum_{\nu=j+1}^{n+1} \dPa{\nu-1}{x}{P_n(x)} \sum_{p=0}^{m} 
\left( \frac{b_{p+\nu,j}}{(p+\nu)!} \dPa{p}{x}{P_m(x)}\right).
\end{multlined}
\end{equation}

In \eqref{gamma with b}, we set $k=0$ and $s=j$ and replacing $r$ and $n$ by 
$j-\nu$ and $m$ respectively to have
\[
\gamma_{m+\nu-j,0}^{(j-\nu,j)} = \sum_{\sigma=0}^{m+\nu-j} 
\frac{b_{\sigma+j,j}}{(\sigma+j)!} \dPa{j-\nu+\sigma}{x}{P_m(x)},
\]
by applying which the first double sum in \eqref{rho jn v3 intermediate} can be 
simplified as
\[
\sum_{\nu=1}^{j}\dPa{\nu-1}{x}{P_n(x)} \sum_{p=0}^{m+\nu-j} 
\frac{b_{p+j,j}}{(p+j)!} \dPa{p-\nu+j}{x}{P_m(x)} = \sum_{\nu=1}^{j} 
\gamma_{m-j+\nu,0}^{(j-\nu,j)} \dPa{\nu-1}{x}{P_n(x)}. 
\]
Hence, \eqref{rho jn v3} is obtained. 
\end{proof}

To obtain the preceding results, we have nowhere assumed the sequence of 
polynomials $\{P_n(x)\}_{n\geqslant 0}$ to be orthogonal and the expressions 
for the $\gamma$-connection coefficients and the $b$-coefficients are, in 
general, not easy to calculate. However, when $\{P_n(x)\}_{n \geqslant 0}$ is an 
orthogonal polynomial sequence, these coefficients are usually explicitly known 
or more likely to be obtainable. In fact, for a non-decreasing, non-negative 
function $w(x)$ in $[a,b]$ which is measurable in the Lebesgue sense, that is, 
all the moments $\int_{a}^b x^n w(x) \mathrm{d}x$ exist and are 
finite\footnote{In the case of $a=-\infty$ or $b=+\infty$, we require that 
$\lim\limits_{x\to -\infty}x^n w(x)$ and $\lim\limits_{x\to +\infty}x^n w(x)$ to 
be finite, respectively, for any positive integer $n$.}, there is an 
orthogonal sequence of polynomials $\{P_n(x)\}_{n\geqslant 0}$ for which 
\begin{equation*}\label{ip ortho}
\langle P_m(x),P_n(x)\rangle_{w} = \int_{a}^b P_m(x)P_n(x) w(x)\mathrm{d}x = 
h_m\delta_{m,n}, \quad m,n = 0,1,2,\ldots
\end{equation*}
where $h_m=\langle P_m(x),P_m(x)\rangle_{w} \neq 0$ for all positive integers 
$m$ and $\delta_{m,n}$ denotes the {\it Kronecker delta symbol}. Since 
$P_n(x)\in L_{w}^2(a,b)$ and $\{P_n(x)\}_{n\geqslant 0}$ spans the vector space 
of polynomials, any polynomial $p(x)$ of degree $m$ can be written as 
\[
p(x) = \sum_{k=0}^m c_{k} P_k(x) \quad \text{with}\quad c_{k} = \frac{\langle 
p(x),P_n(x)\rangle_{w} }{h_k}. 
\]

Particularly, for classical orthogonal polynomial sequences, i.e. Jacobi, 
Laguerre, Hermite, and Bessel polynomials, these $b$- and $\gamma$- connection 
coefficients are well studied \cite{Ismail, RochaMar, Rainville, SRM}. Based on 
these known results, we shall explicitly calculate the $\rho$-coefficients in 
\eqref{gen conv} for the Jacobi and the Laguerre polynomials in the next two 
sections. 

We close this section with the following theorem which shows that a consecutive 
part of the $\rho$-coefficients in \eqref{gen conv} could be exactly zero when 
$\{P_n(x)\}_{n\geqslant 0}$ is a classical orthogonal polynomial sequence. 
However, these zeros are not immediately obvious from Theorem \ref{thm: rho k 
geq m}.

\begin{theorem}\label{thm: classical gen} Let $m,n$ be two nonnegative
integers such that $n\geqslant m$. Suppose $\{P_n(x)\}_{n\geqslant 0}$ is a 
classical polynomial sequence and the interval $(-a,x+a)$ lies within the 
support of the orthogonality measure of $\{P_n(x)\}_{n\geqslant 0}$. When $n 
\geqslant 2m+q+2$, the series coefficients $\rho_{j,n}^m = 0$ for $m+1 \leqslant 
j \leqslant n-m-q-1$, where $q=1,1,0$ and $-1$ for Jacobi, Bessel, Laguerre and 
Hermite polynomials, respectively.
\end{theorem}

\begin{proof} By orthogonality, we have
\begin{equation*}\label{PnPm ortho1}
(t+a)^{\nu} P_n(t) = \sum_{k=\max\{n-\nu,0\}}^{n+\nu} \lambda_{\nu,n}(k) P_k(t),
\end{equation*}
where $\lambda_{\nu,n}(k) = \dfrac{\langle(t+a)^{\nu} P_n(t),P_k(t)\rangle_w} 
{\langle P_k(t),P_k(t)\rangle_w}$, which, together with the Taylor expansion of 
$P_m(x-t)$ about $t=-a$
\begin{equation*}\label{Pm x t}
P_m(x-t) = \sum_{\nu=0}^m \frac{(-1)^{\nu}}{\nu!}\frac{\md^\nu P_m(x+a)}{\md 
(x+a)^\nu} (t+a)^{\nu},
\end{equation*}
gives
\begin{equation}\label{PnPm ortho}
P_m(x-t)P_n(t) = \sum_{\nu=0}^m \frac{(-1)^{\nu}}{\nu!}\frac{\md^\nu P_m(x+a)} 
{\md (x+a)^\nu} \sum_{k=n-\nu}^{n+\nu} \lambda_{\nu,n}(k) P_k(t).
\end{equation}
If, in addition, $\{P_n(x)\}_{n\geqslant 0}$ is a classical sequence, there 
exists a polynomial $\Phi(t)$ of degree at most $2$ and $\xi_{n,\nu}$ such that
\[
\Phi(t) \frac{\md P_{k+1}(t)}{\md t} = \sum_{r=k}^{k+\deg\Phi} \xi_{k,r}  
P_r(t),
\]
where $\xi_{k,k+\deg \Phi} \xi_{k,k} \neq 0 $, for all $n\geqslant 0$. 
Reversely, there are coefficients $ \tilde{\xi}_{k,r} $ such that
\begin{equation*}\label{derivative recurrence}
P_k(t) = \sum_{r=k-q}^{k+1} \tilde{\xi}_{k,r} \frac{\md P_r(t)}{\md t},
\end{equation*}
where $q = \deg\Phi(x)-1$ and $\tilde{\xi}_{k,k-q} \tilde{\xi}_{k,k+1} \neq 0$ 
\cite[Prop. 2.4]{Mar}. This can be deemed as a special case of \eqref{cc genP}. 
In particular, if $\{P_n(x)\}_{n \geqslant 0}$ is the classical orthogonal 
sequence of Jacobi, Bessel, Laguerre and Hermite polynomials, $\deg \Phi(x) = 
2,2,0$, and $1$, respectively.

We integrate \eqref{PnPm ortho} over the interval $(-a,x+a)$ to have
\[
\int_{-a}^{x+a} P_m(x-t) P_n(t) \md t = \sum_{\nu=0}^m \frac{(-1)^{\nu}}{\nu!} 
\frac{\md^\nu P_m(x+a)}{\md (x+a)^\nu}  \sum_{k=n-\nu}^{n+\nu} 
\lambda_{\nu,n}(k) \sum_{r=k-q}^{k+1} \tilde{\xi}_{k,r} \Big(P_{r}(x+a) 
- P_{r}(-a) \Big).
\]
Swapping the order of the last two sums and absorbing the innermost summation into new coefficients 
$\widetilde{\lambda}_{\nu,n}(k)$, we have
\[
\begin{multlined}
\int_{-a}^{x+a} P_m(x-t) P_n(t) \md t = \sum_{\nu=0}^m \frac{ (-1)^{\nu} 
}{\nu!}\frac{\md^\nu P_m(x+a)}{\md (x+a)^\nu} 
\sum_{k=n-\nu-q}^{n+\nu+1} \widetilde{\lambda}_{\nu,n}(k) \Big( 
P_{k}(x+a) - P_{k}(-a) \Big) \\
= \sum_{\nu=0}^m \frac{(-1)^{\nu}}{\nu!}\frac{\md^\nu P_m(x+a)}{\md (x+a)^\nu}  
\sum_{k=n-\nu-q}^{n+\nu+1} \widetilde{\lambda}_{\nu,n}(k) P_{k}(x+a)- 
\sum_{\nu=0}^m \frac{(-1)^{\nu}}{\nu!}S_{\nu,n}\frac{\md^\nu P_m(x+a)}{\md 
(x+a)^\nu},
\end{multlined}
\]
where $\widetilde{\lambda}_{\nu,n}(n-\nu-q)\widetilde{\lambda}_{\nu,n}
(n+\nu+1)\neq 0$ and $S_{\nu,n} = \sum_{k=n-\nu-q}^{n+\nu+1} 
\widetilde{\lambda}_{\nu,n}(k) P_{k}(-a)$. Since $\frac{\md^\nu P_m(x+a)}{\md 
(x+a)^\nu} $ is a polynomial of degree $m-\nu$, there are coefficients 
$\chi_{m,\nu,k}$ such that
\[
\begin{multlined}
\int_{-a}^{x+a} P_m(x-t) P_n(t) \md t = \sum_{\nu=0}^m \frac{(-1)^{\nu}}{\nu!} 
\sum_{k=n-\nu-q}^{n+\nu+1} \widetilde{\lambda}_{\nu,n}(k) 
\sum_{j=k-(m-\nu)}^{k+(m-\nu)}\chi_{m,\nu,k}(j)P_{j}(x+a) \\
- \sum_{\nu=0}^m \frac{(-1)^{\nu}}{\nu!} S_{\nu, n}  \sum_{j=0}^{m-\nu}
\gamma_{m-\nu,j}^{(\nu, 0)} P_j(x+a),
\end{multlined}
\]
where we have applied \eqref{cc genP} to the last sum. After swapping the order 
of the summations, we see there are coefficients $\widetilde{\chi}_{m,n}(j)$ such that 
\[
\begin{multlined}
\int_{-a}^{x+a} P_m(x-t) P_n(t) \md t = \sum_{j=n-m-q}^{m+n+1} 
\widetilde{\chi}_{m,n}(j)P_{j}(x+a) - \sum_{j=0}^{m}\widetilde{\chi}_{m,n}(j) 
P_j(x+a).
\end{multlined}
\]
When $n \geqslant 2m+q+2$, this means that $\rho^m_{j,n} = 0$ for $m+1 \leqslant 
j \leqslant n-m-q-1$.
\end{proof}

\section{Convolution of Jacobi polynomials}

In this section, we derive the $\rho$-coefficients in \eqref{gen conv} based on 
the results of Section 2 for the Jacobi-family, including the subcases of 
Gegenbauer, Legendre, and Chebyshev. To facilitate our discussion, we denote the 
Jacobi-based $\rho$-coefficients by $\rho_{j,n}^{m;(\alpha,\beta)}$ throughout 
this section, that is,  
\begin{equation}\label{gen conv Jac}
\int_{-1}^{x+1} P_m^{(\alpha,\beta)}(x-t) P_n^{(\alpha,\beta)}(t) \mathrm{d}t = 
\sum_{j=0}^{m+n+1} \rho_{j,n}^{m;(\alpha,\beta)} P_j^{(\alpha,\beta)}(x+1),
\end{equation}
which corresponds to \eqref{gen conv} with $a=1$. Here, 
$P_n^{(\alpha,\beta)}(x)$ denotes the Jacobi polynomial of degree $n 
\geqslant 0$ with $\alpha, \beta > -1$. With the most commonly-used 
normalization, which can be found, for example, in \cite[\S 4.2.1]{sze}, it can 
be represented as a terminating hypergeometric function
\begin{equation*}\label{Jac expression}
P_n^{(\alpha,\beta)}(x) = \frac{(\alpha+1)_n}{n!} {}_2F_{1}\left( 
\begin{array}{c}-n, n+\alpha+\beta+1 \\ \alpha+1 \end{array}; 
\frac{1-x}{2}\right), ~~~~~ n \geqslant 0,
\end{equation*}
where $(z)_n$ is the Pochhammer symbol, defined as
\[
(z)_0 \coloneqq 1 \quad \text{and} \quad (z)_n \coloneqq \prod_{\sigma=0}^{n-1} 
(z+\sigma) \text{ for } n \geqslant 1.
\]
Here and in the rest of this paper, we will use the generalized hypergeometric series 
\[
\pFq{p}{q}{a_1,\ldots, a_p}{b_1,\ldots, b_q}{x} \coloneqq \sum_{n=0}^{\infty} 
\frac{(a_1)_n\cdots (a_p)_n}{(b_1)_n\cdots (b_q)_n}\frac{x^n}{n!},
\]
and its detail can be found, for example, in \cite[Ch. 16]{DLMF}. 

The properties of $P_n^{(\alpha,\beta)}$ that we will make use of in the 
rest of this section include its value at $-1$
\begin{equation*}\label{JacMinus1}
P_n^{(\alpha,\beta)}(-1) = \frac{(-1)^n(\beta+1)_n}{n!} 
\end{equation*}
and a symmetry property 
\begin{equation}\label{Jac symm}
P_n^{(\alpha,\beta)}(-x) = (-1)^n P_n^{(\beta, \alpha)}(x).
\end{equation}

The sequence of Jacobi polynomials $\{P_n^{(\alpha,\beta)}(x)\}_{n\geqslant 0}$  
satisfy the orthogonality condition \cite[Ch. 4]{Ismail}
\begin{equation*}
\int_{-1}^1 P_k^{(\alpha,\beta)}\!(x)P_n^{(\alpha,\beta)}\!(x) (1-x)^\alpha 
(1+x)^{\beta} \md x = \frac{2^{\alpha+\beta+1} \Gamma(\alpha+n+1) 
\Gamma(\beta+n+1)}{n!\Gamma(\alpha+\beta+n+1)(\alpha+\beta+2n+1)}\delta_{k,n},
\end{equation*}
for any integers $ n,k\geqslant 0$. Being a member of a classical sequence, 
the $p$-th derivative of a Jacobi polynomial is another Jacobi polynomial with 
shifted parameters 
\begin{equation}\label{Jac der}
\frac{\mathrm{d}^p}{\mathrm{d}x^p} P_n^{(\alpha,\beta)}(x) = 
\frac{(\alpha+\beta+n+1)_p}{2^p}P_{n-p}^{(\alpha+p,\beta+p)}(x), 
\end{equation}
and, in particular, 
\begin{equation}\label{dpJacMinus1}
\left.\frac{\mathrm{d}^p}{\mathrm{d}x^p} P_n^{(\alpha,\beta)}(x)\right|_{x=-1}
= \frac{2^{-p} (-1)^{n+p} (p+\beta +1)_{n-p} (n+\alpha +\beta +1)_p}{(n-p)!}. 
\end{equation}

The properties above allow us to derive a connection formula between the 
derivatives of Jacobi polynomials.

\begin{lemma}\label{lem: cc der form Jacobi} The $p$-th and $q$-th derivatives 
of Jacobi polynomials are linearly connected via 
\begin{equation}\label{cc der form Jacobi}
\frac{\mathrm{d}^p}{\mathrm{d}x^p}P_{n+p}^{(\alpha,\beta)}(x) = \sum_{k=0}^n 
\gamma_{n,k}^{(p, q)}(\alpha,\beta) 
\frac{\mathrm{d}^q}{\mathrm{d}x^q}P_{k+q}^{(\alpha,\beta)}(x) 
\end{equation}
with 
\begin{equation}\label{cc der Jacobi}
\begin{multlined}
\gamma_{n,k}^{(p, q)}(\alpha,\beta) = \frac{(k+p+\alpha+1)_{n-k} (n+p+\alpha 
+\beta+1)_p (n+2p+\alpha +\beta +1)_k}{2^{p-q}(n-k)! (k+q+\alpha+\beta +1)_q 
(k+2 q+\alpha +\beta +1)_k}\\[2mm]
\times\pFq{3}{2}{k-n, \ k+q+\alpha +1,\ k+n+2p+\alpha+\beta+1}{k+p+\alpha 
+1,\ 2k+2q+\alpha+\beta+2}{1}.
\end{multlined}
\end{equation}
\end{lemma}

\begin{proof}
The following connection formula, which can be found in \cite[p.357]{AAR}, 
\cite[Theorem 9.1.1]{Ismail}, or \cite[Eq. (18.18.14)]{DLMF}, relates Jacobi 
polynomials with distinct parameters: 
\begin{equation}\label{cc form Jacobi}
P_n^{(\alpha,\beta)}(x) = \sum_{k=0}^n a_{n,k}^{(\alpha,\beta;\gamma,\delta)} 
P_k^{(\gamma,\delta)}(x),
\end{equation}
where
\begin{equation}\label{cc Jacobi}
a_{n,k}^{(\alpha,\beta;\gamma,\delta)} = \frac{(k+\alpha+1)_{n-k} 
(n+\alpha+\beta+1)_k}{(n-k)! (k+\gamma+\delta+1)_k} 
\pFq{3}{2}{k-n,n+k+\alpha+\beta+1,k+\gamma+1}{k+\alpha+1,2k+\gamma+\delta+2}
{ 1}. 
\end{equation}
Combining \eqref{Jac der} and \eqref{cc form Jacobi}, we have  
\eqref{cc der form Jacobi} with 
\begin{equation}\label{cc der Jacobi 2}
\gamma_{n,k}^{(p, q)}(\alpha,\beta) = 
\frac{(n+p+\alpha+\beta+1)_p}{2^{p-q}(k+q+\alpha+\beta+1)_q} 
a_{n,k}^{(\alpha+p, \beta+p; \alpha+q,\beta+q)}, 
\end{equation}
which, with \eqref{cc Jacobi} substituted in, gives \eqref{cc der Jacobi}.
\end{proof}

\begin{remark} Via \eqref{cc der form Jacobi}, the symmetry property \eqref{Jac 
symm} implies:
\begin{equation*}\label{Jac gamma symmetry}
\gamma_{n,k}^{(p, q)}(\alpha,\beta) = (-1)^{n+k} \ \gamma_{n,k}^{(p, 
q)}(\beta, \alpha). 
\end{equation*}
\end{remark}

One last ingredient we need for deriving $\rho_{k,n}^{m;(\alpha,\beta)}$ is the 
$b$-coefficients in \eqref{xn to Pn gen} for representing the monomial basis in 
terms of Jacobi polynomials. 
\begin{lemma}
The connection coefficients ${b}_{n,k}{(\alpha,\beta)}$ in 
\begin{equation}\label{xn to Jac}
(x+1)^n = \sum_{k=0}^n {b}_{n,k}{(\alpha,\beta)} P^{(\alpha,\beta)}_k(x), 
\end{equation}
are given by
\begin{equation}\label{xn to Jac coeffs}
{b}_{n,k}{(\alpha,\beta)} = 2^n n! (\beta+1)_n \frac{(\alpha+\beta+2k+1) 
\Gamma(\alpha+\beta+k+1)}{(\beta+1)_k \Gamma(\alpha+\beta+n+k+2) (n-k)!}. 
\end{equation}
\end{lemma}
\begin{proof}
See, for example, \cite[(4.2.15)]{Ismail} for the proof.
\end{proof}

\subsection{The Jacobi-based convolution coefficients}

Though we could derive the Jacobi-based convolution coefficients directly from 
the orthogonality of Jacobi polynomials, we opt to find the explicit 
expressions for the $\rho$-coefficients in the expansion of the integral 
\eqref{gen conv Jac} from Theorem \ref{thm: rho k geq m} by substituting in 
\eqref{rho jn v1} and \eqref{rho jn v23} the expressions \eqref{dpJacMinus1}, 
\eqref{cc der Jacobi} and \eqref{xn to Jac coeffs}.

\begin{theorem}\label{thm2 RJac any j v2} 
Let $m$ and $n$ be two positive integers with $n\geqslant m$. The $\rho$ 
coefficients in the expansion \eqref{gen conv Jac} can be expressed as
\begin{subequations}\label{Jac Rjn}
\begin{empheq}{alignat=3}
\label{Jac Rjn j geq n_m}
\rho_{j,n}^{m;(\alpha,\beta)} &= \sum_{\nu=\max(1,|j-n|)}^{m+1} 
\varpi_{j,\nu}^{m,n}(\alpha,\beta) & \quad \text{ for } j\geqslant \max(m+1, 
n-m-1), \\
\label{Jac Rjn j geq m}
\rho_{j,n}^{m;(\alpha,\beta)} &= 0 & \quad \text{ for } m+1\leqslant 
j\leqslant n-m-2, \text{ if } n \geqslant 2m+3,\\
\label{Jac Rjn j leq m} \rho_{j,n}^{m;(\alpha,\beta)} &= \sum_{\nu=1}^{j}  
\varpi_{j,\nu}^{n,m}(\alpha,\beta) + \sum_{\nu=j+1}^{n+1} 
d_{\nu,j,n}^{m}(\alpha,\beta) & \quad \text{ for } 0\leqslant j \leqslant m,
\end{empheq}
where
\begin{equation}\label{Jac varpi}
\begin{multlined}
\varpi_{j,\nu}^{m,n}(\alpha,\beta) = \frac{2(-1)^{m+\nu-1}(j-\nu+\alpha 
+1)_{n-j+\nu} (n+\alpha +\beta+1)_{j-\nu} (\beta +\nu)_{m-\nu+1} (m+\alpha 
+\beta +1)_{\nu-1}}{(j+\alpha +\beta +1)_j(m+1-\nu)!(n-j+\nu)!} \\[2mm]
\times \pFq{3}{2}{j-n-\nu,j+\alpha+1,n+j-\nu+\alpha +\beta +1}{j-\nu 
+\alpha +1,2 j+\alpha +\beta +2}{1},
\end{multlined}
\end{equation}
and
\begin{equation}\label{Jac d}
\begin{multlined}
d_{\nu,j,n}^{m}(\alpha,\beta) = \frac{2(-1)^{m+n+1+\nu} (\alpha +\beta +2j+1) 
(\beta+\nu)(j+\beta +1)_{m-j} (\beta +1)_n (n+1+\alpha +\beta 
)_{\nu-1}}{m!(n+1-\nu)!(\nu-j)!(j+\alpha+\beta +1)_{\nu+1}}\\[2mm]
\times \pFq{4}{3}{1,-m,\beta+\nu +1,m+\alpha +\beta +1}{\nu-j+1,\beta 
+1,j+\alpha +\beta +\nu +2}{1}.
\end{multlined}
\end{equation}
\end{subequations}
\end{theorem}

\begin{proof} 
The zero coefficients given by \eqref{Jac Rjn j geq m} are readily known from
Theorem \ref{thm: classical gen}. 

Following \eqref{rho jn v23} with $a=1$, we have
\begin{subequations}\label{RJac pf1}
\begin{equation}\label{RJac pf1 1}
\rho_{j,n}^{m;(\alpha,\beta)} = \sum_{\nu=\max(1,j-n)}^{m+1} 
\gamma_{n-j+\nu,0}^{(j-\nu,j)}(\alpha,\beta) 
\dP{\nu-1}{x}{P_m^{(\alpha,\beta)}(x)}~~~~~\text{ for } j\geqslant m+1,
\end{equation}
and
\begin{equation}\label{RJac pf1 2}
\begin{multlined}
\rho_{j,n}^{m;(\alpha,\beta)} = 
\sum\limits_{\nu=1}^{j}\gamma_{m-j+\nu,0}^{(j-\nu,j)} 
(\alpha,\beta)\dP{\nu-1}{x}{P_n^{(\alpha,\beta)}(x)} \\
\hspace{1cm}+\sum\limits_{\nu=j+1}^{n+1}\left(\dP{\nu-1}{x}{P_n^{(\alpha,
\beta)}(x)}\sum \limits_{k=0}^{m} \frac{b_{k+\nu,j}(\alpha,\beta)}{(k+\nu)!} 
\dP{k}{x}{P_m^{(\alpha,\beta)}(x)}\right)~~~~~\text{ for }0 \leqslant 
j\leqslant 
m.
\end{multlined}
\end{equation}
\end{subequations}
Combined with \eqref{dpJacMinus1} and \eqref{cc der Jacobi}, \eqref{RJac pf1 
1} gives \eqref{Jac Rjn j geq n_m} and \eqref{Jac varpi}. Similarly, the first 
sum on the right hand side of \eqref{RJac pf1 2} yields the first sum in 
\eqref{Jac Rjn j leq m}. 

To find $d_{\nu,j,n}^{m}(\alpha,\beta)$ of the second sum in \eqref{Jac Rjn j 
leq m}, we first calculate the inner sum of \eqref{RJac pf1 2} to have
\begin{equation}\label{RJac pf1 3}
\begin{multlined}
\sum \limits_{k=0}^{m}\frac{b_{k+\nu,j}(\alpha,\beta)}{(k+\nu)!}\dP{k}{x}{P_m^{
(\alpha,\beta)}(x)} = \frac{2^{\nu} (-1)^{m} (\alpha +\beta +2j+1) \Gamma 
(m+\beta +1)}{m!\Gamma (\nu-j +1)\Gamma (j+\beta +1)} \\
\hspace{.5cm} \times \frac{\Gamma (\beta +\nu +1)\Gamma (j+\alpha +\beta 
+1)}{\Gamma (\beta +1) \Gamma (j+\alpha +\beta +\nu +2)} \sum \limits_{k=0}^{m} 
\frac{(-m)_k(\beta +\nu +1)_k (m+\alpha +\beta +1)_k}{ (\beta +1)_k(-j+\nu +1)_k 
(j+\alpha +\beta +\nu +2)_k}.
\end{multlined}
\end{equation}
Multiplying by $\dP{\nu-1}{x}{P_n^{(\alpha,\beta)}(x)}$ the factors in 
\eqref{RJac pf1 3} that are independent of index $k$ and simplifying yields
\begin{equation*}
\begin{multlined}
\frac{2^{\nu} (-1)^{m} (\alpha +\beta +2 j+1) \Gamma (m+\beta +1) \Gamma (\beta 
+\nu +1)\Gamma (j+\alpha +\beta +1)}{m!\Gamma (\nu-j +1)\Gamma (j+\beta +1) 
\Gamma (\beta +1) \Gamma (j+\alpha +\beta +\nu +2)}\dP{\nu-1}{x}{ 
P_n^{(\alpha,\beta)}}\\[2mm]
=(-1)^{m+n+1+\nu} \frac{ 2 (\alpha +\beta +2 j+1) (\beta+\nu)(j+\beta 
+1)_{m-j} (\beta +1)_n (n+1+\alpha +\beta)_{\nu-1}} 
{m!(n+1-\nu)!(\nu-j)!(j+\alpha +\beta +1)_{\nu+1}}.
\end{multlined}
\end{equation*}
The expression of $d_{\nu,j,n}^{m}(\alpha,\beta)$ then follows from the last 
two equations and the fact that the $k$-sum in \eqref{RJac pf1 3} can be 
concisely written as a generalized hypergeometric series:
\[
\sum_{k=0}^m \frac{(-m)_k(\beta +\nu +1)_k (m+\alpha +\beta +1)_k}{(\beta 
+1)_k(-j+\nu +1)_k (j+\alpha +\beta +\nu +2)_k} = \pFq{4}{3}{1,-m,\beta 
+\nu +1,m+\alpha +\beta +1}{\nu-j+1,\beta +1,j+\alpha +\beta +\nu +2}{1}.
\]
\end{proof}

\begin{figure}[ht]
\centering
\includegraphics[scale=0.6]{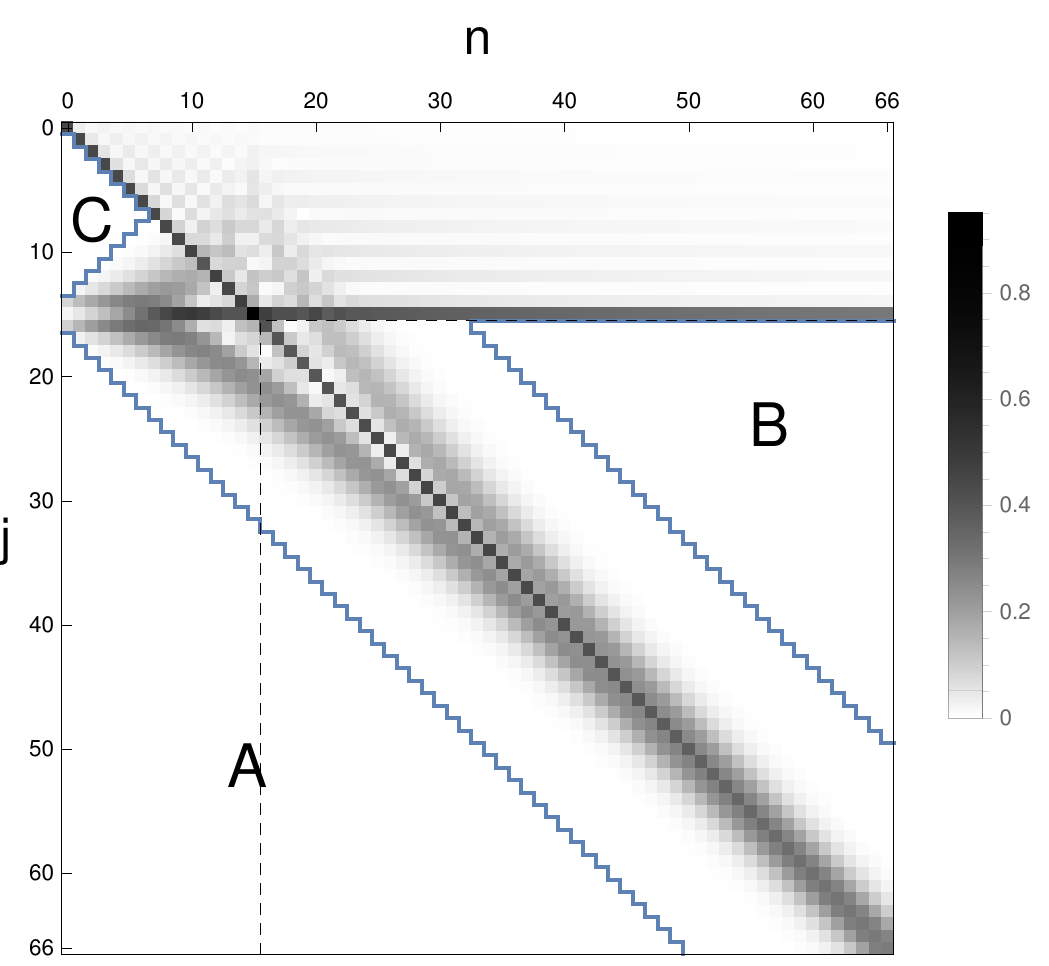}
\caption{The magnitude plot of the Jacobi coefficients $\rho_{j,n}^{15;(2.5, 
1.5)}$.}\label{FIG:jac}
\end{figure}

We calculate the coefficients $\rho_{j,n}^{m;(\alpha, \beta)}$ with $\alpha = 
2.5$ and $\beta=1.5$ for $m = 15$ using Theorem \ref{thm2 RJac any j 
v2}\footnote{In this example and the examples in the remainder of this paper, 
the calculation is carried out in {\sc Mathematica}.} and show their 
magnitudes for $0 \leqslant j, n \leqslant 66$ in Figure \ref{FIG:jac}. That 
is, the $n$-th column in this matrix plot corresponds to $\rho_{j,n}^{15;(2.5, 
1.5)}$. There are three regions of exact zeros which are indicated by solid 
lines. The coefficients $\rho_{j,n}^{15;(2.5, 1.5)}$ are exact zeros in Region A 
simply because the convolution of $P_{15}^{(2.5,1.5)}$ and $P_n^{(2.5,1.5)}$ is 
a polynomial of degree $n+16$. The zeros in Region B correspond to \eqref{Jac 
Rjn j geq m}. Finally, we note that $m \geqslant 2n+3$ in Region C and by 
swapping the roles of $m$ and $n$ we see the exact zeros in Region C, again, 
from \eqref{Jac Rjn j geq m}.

In \cite[Th. 4.8]{XL}, it is shown that the coefficients $\rho_{j,n}^{m;(\alpha, 
\beta)}$ are symmetric up to a scaling for $j,n\geqslant m+1$: 
\begin{equation}\label{symP}
\rho_{n, j}^{m;(\alpha,\beta)} = (-1)^{j+n}\frac{(\alpha+\beta+2n+1)(\alpha 
+1)_j (\beta +1)_j \Big( (\alpha+\beta+1)_n \Big)^2}{(\alpha+\beta+2j+1) 
(\alpha+1)_n (\beta+1)_n \Big( (\alpha+\beta+1)_j \Big)^2} \rho_{j, 
n}^{m;(\alpha,\beta)},
\end{equation}
which can be shown from \eqref{Jac varpi} with some tedious and lengthy work. 
However, this symmetry property is readily seen for the symmetric Jacobi case 
where $\alpha = \beta$ and we show this in the next subsection. In Figure 
\ref{FIG:jac}, the entries that satisfy this symmetry property are those in 
the lower right part that is bordered by the dashed lines. 

{\it Bateman's formula} for the expansion of Jacobi polynomials with two 
variables is well known (see, for example, \cite[Theorem 4.3.3]{Ismail}). 
In passing, we obtain the following proposition where we show the 
binomial-type tensor product expansion of $P_m^{(\alpha,\beta)}(x-t)$ in 
$P_j^{(\alpha,\beta)}(x+1)$ and $P_k^{(\alpha,\beta)}(t)$. This way, the 
variables $x$ and $t$ in a Jacobi-polynomial-based {\it difference kernel} 
\cite[p. 37]{lin} become detached.
\begin{proposition}\label{thm: binom type Jac}
For the $m$-th degree Jacobi polynomial $P_m^{(\alpha,\beta)}(x)$,
\begin{equation}\label{Jac Binom type form}
P_m^{(\alpha,\beta)}(x-t) = \sum_{k=0}^m\sum_{j=0}^{m-k} 
c_{m-k,j}^{m,(\alpha,\beta)} P_j^{(\alpha,\beta)}(x+1) P^{(\alpha,\beta)}_k(t),
\end{equation}
where
\begin{equation}\label{Jac Binom type form coeffs}
\begin{multlined}
c_{m-k,j}^{m,(\alpha,\beta)} = \sum_{\nu=k}^{m-j}   
\frac{(-1)^{\nu}(\alpha+\beta+2k+1)}{(\beta+\nu+1)_{k-\nu} 
(\alpha+\beta+k+1)_{\nu+1} (\nu-k)!} \\
\hspace{3cm} \times \frac{(j+\nu+\alpha+1)_{m-\nu-j}(m+\alpha+\beta+1)_\nu 
(m+\nu+\alpha+\beta+1)_j}{(m-\nu-j)! (j+\alpha+\beta+1)_j} \\ 
\hspace{3cm} \times \pFq{3}{2}{j-m+\nu,j+\alpha+1,j+m+\nu+\alpha+\beta 
+1}{j+\nu+\alpha+1,2j+\alpha+\beta+2}{1}.
\end{multlined}
\end{equation}
\end{proposition}

\begin{proof}
By Taylor expansion of $P_m^{(\alpha,\beta)}(x-t) $ about $t=-1$, we have
\[
P_m^{(\alpha,\beta)}(x-t) = \sum_{\nu=0}^m\frac{(-1)^{\nu}}{\nu!} 
\frac{\md^\nu P_m^{(\alpha,\beta)}(x+1)}{\md(x+1)^{\nu}}(t+1)^\nu.
\]
Using \eqref{cc der form Jacobi} and \eqref{xn to Jac}, the latter equation 
becomes
\[
P_m^{(\alpha,\beta)}(x-t) = \sum_{\nu=0}^m\frac{(-1)^{\nu}}{\nu!} \left(
\sum_{j=0}^{m-\nu} \gamma_{m-\nu,j}^{(\nu,0)} (\alpha,\beta) 
P_j^{(\alpha,\beta)}(x+1)\right) \left(\sum_{k=0}^\nu 
b_{\nu,k}{(\alpha,\beta)} P_k^{(\alpha,\beta)}(t)\right).
\]
We swap the order of the $k$- and the $\nu$-summations and then that of 
the $\nu$- and the $j$-summations to obtain \eqref{Jac Binom type form} with
\[
c_{m-k,j}^{m,(\alpha,\beta)} = \sum_{\nu=k}^{m-j} \frac{(-1)^{\nu}}{\nu!}
b_{\nu,k}{(\alpha,\beta)}\gamma_{m-\nu,j}^{(\nu,0)}(\alpha,\beta). 
\]
Substituting in the expressions of $b_{\nu,k}{(\alpha,\beta)}$ and 
$\gamma_{m-\nu,j}^{(\nu, 0)}(\alpha,\beta)$, given by \eqref{xn to Jac coeffs} 
and \eqref{cc der Jacobi} respectively, leads to \eqref{Jac Binom type form 
coeffs}. 
\end{proof} 


\subsection{Symmetric Jacobi polynomials}
In this subsection, we give in Corollary \ref{cor RSymJac any j v2} the 
convolution coefficients $\rho_{k,n}^{m;(\alpha,\beta)}$ for the Jacobi 
polynomials with $\alpha=\beta$. These coefficients could be obtained 
from Theorem \ref{thm2 RJac any j v2} by simply setting $\beta = \alpha$. 
However, a lengthy simplification is necessary in order to obtain exactly what 
is given in Corollary \ref{cor RSymJac any j v2}. The route we take is to obtain 
the explicit expressions for $\gamma_{n,k}^{(p, q)}(\alpha,\alpha)$, 
$b_{n,k}(\alpha, \alpha)$, and $\left.\frac{\mathrm{d}^p}{\mathrm{d}x^p} 
P_n^{(\alpha,\alpha)}(x)\right|_{x=-1}$, from which we derive 
$\varpi_{j,\nu}^{m,n}(\alpha, \alpha)$ and $d_{\nu,j,n}^{m}(\alpha, \alpha)$ 
using \eqref{RJac pf1}. 

\begin{lemma}
When $n-k$ is even, 
\begin{subequations}\label{cc der SymJacobi}
\begin{equation}
\gamma_{n,k}^{(p, q)}(\alpha, \alpha) = \ds \frac{ 
\left(p-q\right)_{\frac{n-k}{2}}(\alpha+k+q+\frac{1}{2}) \Gamma (k+q+2 
\alpha +1) \Gamma (n+p+\alpha +1) \Gamma \left(\frac{k+n+1}{2}+p+\alpha 
\right)}{2^{q-p}\left(\frac{n-k}{2} \right)! \Gamma(k+q+\alpha+1 ) \Gamma 
(n+p+2 \alpha +1) \Gamma \left(\frac{k+n+3}{2} +q+\alpha \right)},
\end{equation}
and 
\begin{equation}
\gamma_{n,k}^{(p, q)}(\alpha, \alpha)=0
\end{equation}
\end{subequations}
otherwise.
\end{lemma}
\begin{proof} The connection coefficients in \eqref{cc form Jacobi} with 
$\beta=\alpha$ and $\delta=\gamma$ can be found in \cite[Theorem 7.1.4]{AAR}:
\begin{equation}\label{sym a even}
a_{n,k}^{(\alpha+p,\alpha+p;\alpha+q,\alpha+q)}= 
\frac{(p-q)_{\frac{n-k}{2}}\left(q+\alpha +\frac{3}{2}\right)_k (2(q+\alpha 
)+1)_k (p+\alpha +1)_n \left(p+\alpha +\frac{1}{2}\right)_{\frac{k+n}{2}}}
{\left(\frac{n-k}{2}\right)! \left(q+\alpha +\frac{1}{2}\right)_k (q+\alpha 
+1)_k (2 (p+\alpha )+1)_n \left(q+\alpha +\frac{3}{2}\right)_{\frac{k+n}{2}}},
\end{equation}
when $n-k$ is even, while
\[ 
a_{n,k}^{(\alpha+p,\alpha+p;\alpha+q,\alpha+q)} = 0,
\] 
when $n-k$ is odd, implied by $P_n^{(\alpha,\alpha)}(-x)=(-1)^n 
P_n^{(\alpha,\alpha)}(x)$ \cite[Theorem 7.1.4]{AAR}. Now the relation \eqref{cc 
der Jacobi 2} between the $a$- and the $\gamma$-coefficients readily show that 
$\gamma_{n,k}^{(p, q)}(\alpha, \alpha) = 0$ when $n-k$ is odd. For the case 
where $n-k$ is even, we substitute \eqref{sym a even} in \eqref{cc der Jacobi 2} 
and simplify using the Legendre duplication formula \cite[Eq. (5.5.5)]{DLMF} to 
obtain \eqref{cc der SymJacobi}.
\end{proof}

\begin{corollary}\label{cor RSymJac any j v2} Let $m$ and $n$ be two positive 
integers with $n\geqslant m$ and suppose $\alpha>-1$ to be nonzero. For 
$\alpha = \beta$, the $\rho$-coefficients in the expansion \eqref{gen conv Jac} 
become
\begin{subequations}\label{SymJac Rjn}
\begin{empheq}{alignat=3}
\label{SymJac Rjn j geq m}
\rho_{j,n}^{m;(\alpha, \alpha)} &= \sum_{\nu=\max(1, |j-n|)}^{m+1} 
\varpi_{j,\nu}^{m,n}(\alpha, \alpha) & \quad \text{ for } j\geqslant 
\max(m+1,n-m-1), \\
\label{SymJac Rjn j geq m2}
\rho_{j,n}^{m;(\alpha, \alpha)} &=0 & \quad \text{ for } m+1 \leqslant j 
\leqslant n-m-2, \text{ if } n \geqslant 2m+3,\\
\label{SymJac Rjn j leq m} 
\rho_{j,n}^{m;(\alpha, \alpha)} &= \sum_{\nu=1}^{j} 
\varpi_{j,\nu}^{n,m}(\alpha, 
\alpha) + \sum_{\nu=j+1}^{n+1} d_{\nu,j,n}^{m}(\alpha, \alpha) & \quad \text{ 
for } 0\leqslant j \leqslant m,
\end{empheq}
where
\begin{equation}\label{varpi sym Jac even}
\begin{multlined}
\varpi_{j,\nu}^{m,n}(\alpha, \alpha) = 
\frac{2(-1)^{m+\nu+1}(\alpha+\nu)_{m+1-\nu}(m+2\alpha+1)_{\nu-1}(n+2 \alpha 
+1)_{j-\nu}}{(m-\nu+1)! \left(\frac{n+\nu-j}{2} \right)!} \\
\times \frac{(-\nu)_{\frac{n+\nu-j}{2}} 
\left(j+\alpha-\nu+\frac{1}{2}\right)_{\frac{n+\nu-j}{2}} (j+\alpha -\nu 
+1)_{n+\nu-j}}{(j+2\alpha+1)_j \left(j+\alpha
+\frac{3}{2}\right)_{\frac{n+\nu-j}{2}}(2 j+2\alpha-2\nu +1)_{n+\nu-j}}
\end{multlined}
\end{equation}
when $n+\nu-j$ is even and
\begin{equation}\label{varpi sym Jac odd}
\varpi_{j,\nu}^{m,n}(\alpha, \alpha) = 0
\end{equation}
otherwise. Here,
\begin{equation}\label{JacSym d}
\begin{multlined}
d_{\nu,j,n}^{m}(\alpha, \alpha) = \frac{2(-1)^{m+n+1+\nu} (2\alpha +2j+1) 
(\alpha +\nu)(j+ \alpha +1)_{m-j} (\alpha +1)_n 
(n+1+2\alpha)_{\nu-1}}{m!(n+1-\nu)!(\nu-j)!(j+2\alpha +1)_{\nu+1}}\\[2mm]
\times \pFq{4}{3}{1,-m, \alpha +\nu +1,m+2 \alpha +1}{\nu-j+1, \alpha +1,j+2 
\alpha +\nu +2}{1}.
\end{multlined}
\end{equation}
\end{subequations}
\end{corollary}

\begin{proof}
The $\rho$-coefficients in \eqref{SymJac Rjn} inherit from \eqref{Jac Rjn}. 

From \eqref{dpJacMinus1}, we have
\begin{equation}\label{dpJacMinus1 sym}
\left. \frac{\mathrm{d}^p}{\mathrm{d}x^p} P_n^{(\alpha,\alpha)}(x)\right|_{x=-1}
= \frac{2^{-p} (-1)^{n+p} (p+\alpha +1)_{n-p} (n+2\alpha +1)_p}{(n-p)!}. 
\end{equation}

Substituting \eqref{cc der SymJacobi} and \eqref{dpJacMinus1 sym} into 
\eqref{RJac pf1 1}, we have \eqref{varpi sym Jac even} and \eqref{varpi sym Jac 
odd} after some algebraic simplifications.

Replacing $\beta$ by $\alpha$ in \eqref{Jac d} gives \eqref{JacSym d}.
\end{proof}

The following proposition instantiates the symmetry property for the symmetric
Jacobi-based coefficients, which can be easily derived from Corollary \ref{cor 
RSymJac any j v2}.

\begin{proposition}\label{prop: sym Rsym}
For $j,n \geqslant m+1$, the $\rho$-coefficients in \eqref{gen conv Jac} with 
$\beta=\alpha$ satisfy
\begin{equation}\label{rho ratio}
\rho_{j,n}^{m;(\alpha, \alpha)} = (-1)^{j+n} \frac{(2\alpha+2j+1) \left((2 
\alpha+1)_j\right)^2 \left((\alpha +1)_n\right)^2}{(2\alpha+2n+1) \left((\alpha 
+1)_j\right)^2 \left((2\alpha+1)_n\right)^2} \rho_{n,j}^{m;(\alpha, \alpha)}.
\end{equation}
\end{proposition} 
\begin{proof}

When $(n+\nu-j)$ is even, so is $(j+\nu-n)$ and \eqref{varpi sym Jac even} gives
\begin{equation}\label{varpi ratio}
\frac{\varpi_{j,\nu}^{m,n}(\alpha, \alpha)}{\varpi_{n,\nu}^{m,j}(\alpha,\alpha)} 
= (-1)^{j+n}\frac{(2\alpha+2j+1) \left((2\alpha+1)_j\right)^2
\left((\alpha+1)_n\right)^2}{(2\alpha+2n+1) \left((\alpha+1)_j\right)^2 
\left((2\alpha+1)_n\right)^2},
\end{equation}
which is independent from $\nu$. If $|n-j|\leqslant m+1$, \eqref{varpi ratio} and \eqref{SymJac 
Rjn j geq m} imply \eqref{rho ratio}.

If $n \geqslant 2m+3$, for $m+1 \leqslant j \leqslant n-m-2$ we have
$\rho_{j,n}^{m;(\alpha, \alpha)} = 0$, as indicated by \eqref{SymJac Rjn j geq 
m2}. Also, $j \leqslant n-m-2$ suggests $n \geqslant j+m+2$, in which case 
$\rho_{n,j}^{m;(\alpha, \alpha)} = 0$. Therefore, \eqref{rho ratio} holds too 
for $m+1 \leqslant j \leqslant n-m-2$.

Hence, \eqref{rho ratio} is true for $j, n \geqslant m+1$.
\end{proof}

\begin{figure}[ht]
\centering
\includegraphics[scale=0.6]{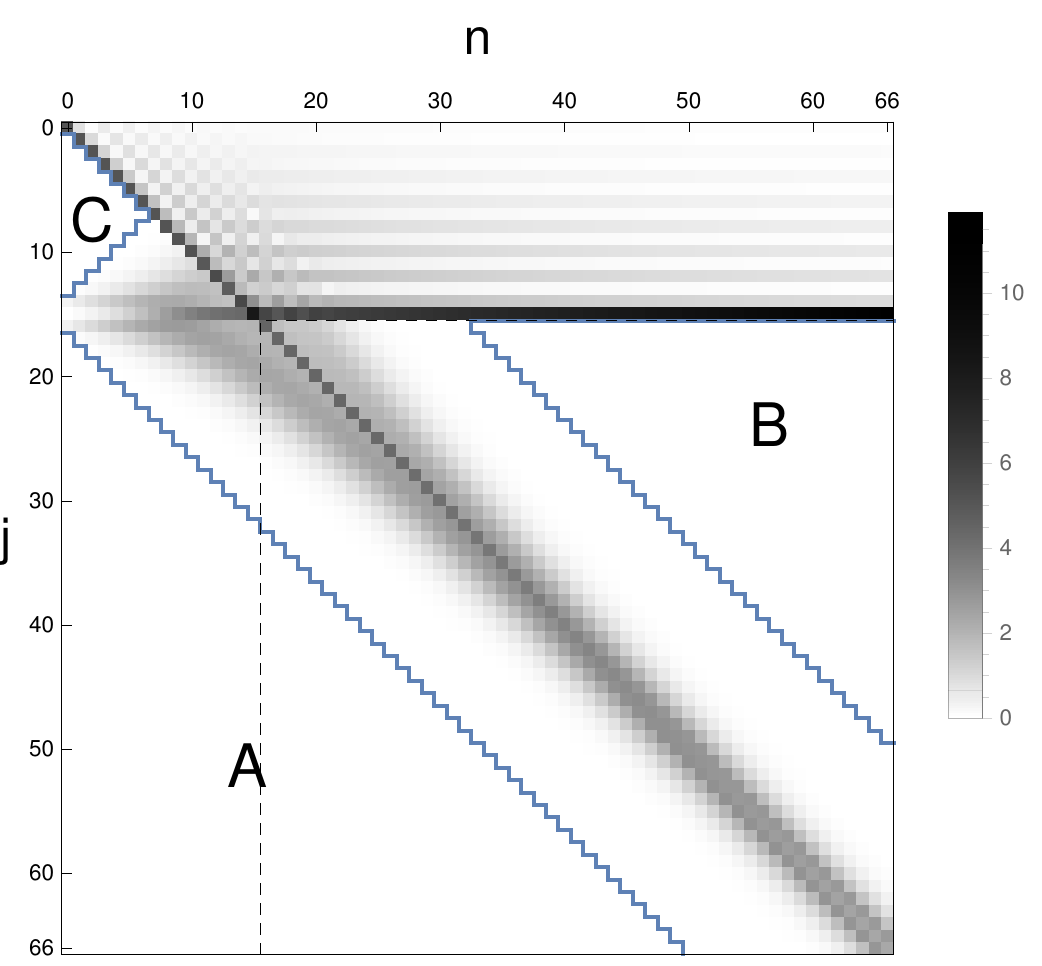}
\caption{The magnitude plot of the symmetric Jacobi coefficients 
$\rho_{j,n}^{15;(2.5, 2.5)}$.}\label{FIG:symjac}
\end{figure}

Figure \ref{FIG:symjac} shows the magnitudes of the coefficients 
$\rho_{j,n}^{m;(\alpha, \alpha)}$ with $\alpha = 2.5$ and $m = 15$. The regions 
A, B, and C with exact zeros are inherited from those of the Jacobi-based 
convolution coefficients. The symmetric coefficients are bordered by the dashed 
lines.



\subsubsection{Gegenbauer case}

Gegenbauer polynomials $C_n^{(\lambda)}$ are symmetric Jacobi 
polynomials with a different normalization:
\begin{equation}\label{Geg to Jac}
C_n^{(\lambda)}(x) = 
\frac{\left(2\lambda\right)_n}{\left(\lambda+\frac{1}{2}\right)_n} 
P_{n}^{(\lambda-\frac{1}{2},\lambda-\frac{1}{2})}(x) ,\ n\geqslant 0, 
\end{equation}
with $\lambda>-\frac{1}{2}$ and $\lambda\neq0$. If we denote by 
$\widetilde{\rho}_{k,n}^{m;(\lambda)}$ the series coefficients of the 
convolution of two Gegenbauer polynomials, that is,
\begin{equation}\label{gen conv Geg}
\int_{-1}^{x+1} C_m^{(\lambda)}(x-t)  C_n^{(\lambda)}(t) \mathrm{d}t = 
\sum_{k=0}^{m+n+1} \widetilde{\rho}_{j,n}^{m;(\lambda)} C_{j}^{(\lambda)}(x+1), 
\end{equation}
the relation \eqref{Geg to Jac} gives
\begin{equation}\label{geg coeff scaling}
\widetilde{\rho}_{j,n}^{m;(\lambda)} = \frac{\left(\lambda + 
\frac{1}{2}\right)_j (2 \lambda 
)_m (2 \lambda)_n}{(2 \lambda )_j \left(\lambda +\frac{1}{2}\right)_m 
\left(\lambda +\frac{1}{2}\right)_n} \rho_{j,n}^{m;(\lambda-1/2;\lambda-1/2)}, 
\end{equation}
where $\rho_{j,n}^{m;(\lambda-1/2;\lambda-1/2)}$ are the coefficients given in 
Corollary \ref{cor RSymJac any j v2}. Combining Corollary \ref{cor RSymJac any 
j v2} and \eqref{geg coeff scaling} leads to the following corollary, the proof 
of which is omitted. 
\begin{corollary}
Let $m$ and $n$ be two positive integers with $n\geqslant m$ and suppose 
$\lambda> -\frac{1}{2}$ to be nonzero. The $\widetilde{\rho}$-coefficients in the expansion 
\eqref{gen conv Geg} can be expressed as
\begin{empheq}{alignat=3}
\widetilde{\rho}_{j,n}^{m;(\lambda)} &= \sum_{\nu=\max(1, |j-n|)}^{m+1} 
\widetilde{\varpi}_{j,\nu}^{m,n} (\lambda) & \quad \text{ for } j\geqslant 
\max(m+1,n-m-1), \nonumber \\
\widetilde{\rho}_{j,n}^{m;(\lambda)} &=0 & \quad \text{ for } m+1 \leqslant j 
\leqslant n-m-2, \text{ if } n \geqslant 2m+3, \nonumber \\
\widetilde{\rho}_{j,n}^{m;(\lambda)} &= \sum_{\nu=1}^{j} 
\widetilde{\varpi}_{j,\nu}^{n,m}(\lambda) + \sum_{\nu=j+1}^{n+1} 
\widetilde{d}_{\nu,j,n}^{m}(\lambda) & \quad \text{ for } 0\leqslant j 
\leqslant m, \nonumber
\end{empheq}
where
\begin{equation*} 
\widetilde{\varpi}_{j,\nu}^{m,n}(\lambda) = \frac{(j+\lambda) (-1)^{m+\nu+1} 
(\lambda)_{\nu-1} (2\lambda + 2\nu - 2)_{m-\nu+1}(-\nu)_{\frac{-j+n+\nu}{2}}}
{2 (m-\nu+1)! \left(\frac{-j+\nu+n}{2}\right)! \left(\lambda+\frac{j+n-\nu}{2} 
\right)_{\nu+1}}
\end{equation*}
when $n+\nu-j$ is even and $\widetilde{\varpi}_{j,\nu}^{m,n}(\lambda) = 0$ 
otherwise. Here, 
\begin{equation*}
\begin{multlined}
\widetilde{d}_{\nu,j,n}^{m}(\lambda) = \frac{2(-1)^{m+\nu+n+1}(j+\lambda)(2 
\lambda+2\nu-1)(2\lambda)_m(j+2\lambda+\nu+1)_{-j+n-2}}{m!(\nu-j)!(-\nu
+n+1)!}\\[2mm]
\times \pFq{4}{3}{1,-m,m+2\lambda,\lambda+\nu+\frac{1}{2}}{\lambda+\frac{1}{2}, 
-j+\nu+1,j+2\lambda+\nu+1}{1}.
\end{multlined}
\end{equation*}
\end{corollary}

\subsubsection{Legendre case}\label{subsec:Leg}
Legendre polynomials $P_n^{(0,0)}(x)$ are the symmetric Jacobi polynomials with 
$\alpha=\beta=0$ or, equivalently, the special case of Gegenbauer polynomials 
with $\lambda=1/2$. We show in the following corollary that the convolution 
coefficients of Legendre polynomials become significantly simpler than those of 
symmetric Jacobi or Gegenbauer.

\begin{corollary}\label{cor Leg any j v2} Let $m$ and $n$ be two positive 
integers with $n \geqslant m$. The coefficients $\rho_{j,n}^{m;(0,0)}$ in 
\eqref{gen conv Jac} can be expressed as
\begin{empheq}{alignat=3}
\rho_{j,n}^{m;(0,0)} &= \sum_{\nu=\max(1, |j-n|)}^{m+1} 
\varpi_{j,\nu}^{m,n}(0,0) & 
\quad \text{ for } j\geqslant \max(m+1,n-m-1), \nonumber \\
\rho_{j,n}^{m;(0,0)} &=0 & \quad \text{ for } m+1 \leqslant j \leqslant 
n-m-2, \text{ if } n \geqslant 2m+3, \nonumber \\
\rho_{j,n}^{m;(0,0)} &= \sum_{\nu=1}^{j} \varpi_{j,\nu}^{n,m}(0,0) + 
\sum_{\nu=j+1}^{n+1} d_{\nu,j,n}^m  & \quad \text{ for } 0\leqslant j 
\leqslant m, \nonumber
\end{empheq}
where
\begin{subequations}
\begin{equation}\label{Leg inner coef}
\varpi_{j,\nu}^{m,n}(0,0) =  \frac{(-1)^{m+\nu+1}(2j+1)(m+\nu-1)!(-\nu 
)_{\frac{n-j+\nu}{2}}}{4^{\nu}(\nu-1)! (m-\nu+1)! 
\left(\frac{n-j+\nu}{2} \right)! \left(\frac{n+j-\nu+1}{2}\right)_{\nu+1}}
\end{equation}
for even $n+\nu-j$ and $\varpi_{j,\nu}^{m,n}(0,0) = 0$ otherwise. Here, 
\begin{equation}\label{Leg d}
{d}_{\nu,j,n}^{m}(0,0) = \frac{\sqrt{\pi}(2j+1)\nu 2^{j+m-\nu}(-1)^{m+\nu 
+n+1} (n+\nu-1)! \left(\frac{-j-m+\nu+1}{2} \right)_{j+m} \left(\frac{j-m+\nu 
+2}{2} \right)_m} {(n-\nu+1)!(j+m+\nu)! \Gamma \left(\frac{-j+m+\nu+2}{2} 
\right) \Gamma \left(\frac{j+m+\nu+3}{2} \right)}.
\end{equation}
\end{subequations}
\end{corollary}
\begin{proof} Since $P_n(x) = P_n^{(0,0)}(x)$, \eqref{Leg inner coef} 
can be obtained by setting $\alpha=0$ in \eqref{varpi sym Jac even} and 
simplifying using the Legendre duplication formula.

Setting $\alpha=0$ in \eqref{JacSym d}, we have
\begin{equation*}
{d}_{\nu,j,n}^m(0,0) = \frac{2(-1)^{\nu+m+n+1}v(2j+1) 
(n+\nu-1)!}{(n-\nu+1)! 
(\nu-j)! (j+\nu+1)!}\pFq{3}{2}{-m,m+1,\nu+1}{-j+\nu+1,j+\nu+2}{1},
\end{equation*}
where the generalized hypergeometric series ${}_3F_2$ can be represented in 
terms of Gamma functions using the Whipple's sum \cite[Eq.~(16.4.7)]{DLMF}. 
Further algebraic simplification using the Legendre duplication formula gives 
\eqref{Leg d}.
\end{proof}

The following theorem shows that the symmetry property \eqref{symP} holds for 
all $j,n \geqslant 0$ in the Legendre case, which is difficult to see directly 
from Theorem \ref{cor Leg any j v2}. Our proof employs the fact that the 
Legendre polynomials are $L^2$ orthogonal on $[-1, 1]$.

\begin{theorem} For any $m,n,j \geqslant 0$, the coefficients 
$\rho_{j,n}^{m;(0,0)}$ in \eqref{gen conv Jac} satisfy
\begin{equation}\label{sym Leg rho}
\rho_{j,n}^{m;(0,0)} = (-1)^{n+j}\frac{2j+1}{2n+1} \rho_{n,j}^{m;(0,0)}.
\end{equation}
\end{theorem}
\begin{proof} 
As indicated in \eqref{conv1}, $x \in [-2, 0]$ in \eqref{gen conv Jac}. By 
letting $y = x+1$, we have
\begin{equation*}\label{gen conv leg}
\int_{-1}^y P_m^{(0,0)}(y-t-1) P_n^{(0,0)}(t) \mathrm{d}t = 
\sum_{k=0}^{m+n+1} \rho_{k,n}^{m;(0,0)} P_k^{(0,0)}(y),
\end{equation*}
where $y \in [-1, 1]$. Since the Legendre polynomials are $L^2$ orthogonal, 
i.e. for $j, n \geqslant 0$,
\begin{equation*}
\int_{-1}^1 P^{(0,0)}_j(y) P^{(0,0)}_k(y) \mathrm{d}y = \frac{2}{2j+1} 
\delta_{j,k},
\end{equation*}
we have
\begin{equation}\label{projection}
\int_{-1}^1 P^{(0,0)}_j(y) \int_{-1}^y P^{(0,0)}_{m}(y-t-1)P^{(0,0)}_n(t) 
\mathrm{d}t \mathrm{d}y = \frac{2}{2j+1}\rho_{j,n}^{m;(0,0)}
\end{equation}
for $0 \leqslant j \leqslant m+n+1$.

Now, we swap the order of integration to have
\begin{equation*}
\int_{-1}^1 P^{(0,0)}_j(y) \int_{-1}^y P^{(0,0)}_{m}(y-t-1)P^{(0,0)}_n(t) 
\mathrm{d}t \mathrm{d}y =
\int_{-1}^1 P^{(0,0)}_n(t) \int_{t}^1 P^{(0,0)}_j(y) P^{(0,0)}_{m}(y-t-1) 
\mathrm{d}y \mathrm{d}t,
\end{equation*}
where $t \in [-1, 1]$. Applying the changes of variables $t = -T$ and 
$y=-Y$ gives
\begin{equation}\label{swap integration}
\begin{multlined}
\int_{-1}^1 P^{(0,0)}_j(y) \int_{-1}^y P^{(0,0)}_{m}(y-t-1)P^{(0,0)}_n(t) 
\mathrm{d}t \mathrm{d}y \\
\hspace{4cm} = (-1)^{n+j} \int_{-1}^1 P^{(0,0)}_n(T) \int_{-1}^{T} 
P^{(0,0)}_j(Y) P^{(0,0)}_{m}(T-Y-1) \mathrm{d}Y \mathrm{d}T,
\end{multlined}
\end{equation}
where we have used $P_n(y)=(-1)^n P_n(-y)$. 

Recognizing the inner integral in \eqref{swap integration} as the convolution of 
$P_j(T)$ and $P_m(T)$ with $T \in [-1, 1]$, we can replace it by its series 
representation 
\begin{equation}\label{gen conv leg 2}
\begin{multlined}
\int_{-1}^1 P^{(0,0)}_j(y) \int_{-1}^y P^{(0,0)}_{m}(y-t-1)P^{(0,0)}_n(t) 
\mathrm{d}t \mathrm{d}y \\ 
\hspace{2cm} = (-1)^{n+j} \sum_{k=0}^{j+m+1} \rho_{k,j}^{m;(0,0)} \int_{-1}^1 
P^{(0,0)}_n(T) P^{(0,0)}_k(T) \mathrm{d}T = (-1)^{n+j}\frac{2}{2n+1} 
\rho_{n,j}^{m;(0,0)},
\end{multlined}
\end{equation}
where the second equality follows from the orthogonality. Finally, 
\eqref{projection} and \eqref{gen conv leg 2} leads to \eqref{sym Leg rho}.
\end{proof}

\begin{figure}[ht]
\centering
\includegraphics[scale=0.6]{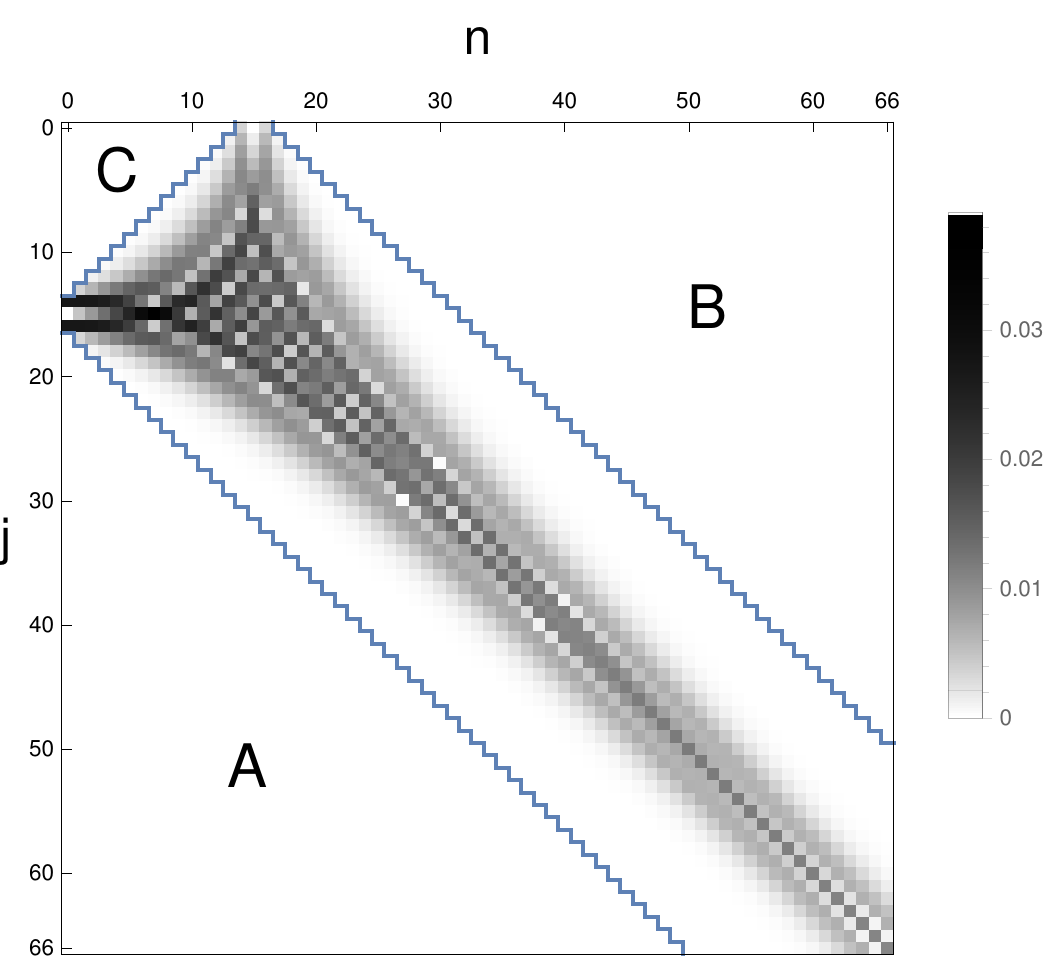}
\caption{The magnitude plot of the Legendre coefficients 
$\rho_{j,n}^{15;(0,0)}$.}\label{FIG:leg}
\end{figure}

The Legendre-based symmetry property \eqref{sym Leg rho} suggests extended 
regions of exact zeros in the magnitude plot, as seen in Figure \ref{FIG:leg}. 
Now, the non-zero coefficients are confined in a tilted rectangle-shaped band 
surrounded by zeros in Regions A, B, and C.

\subsubsection{Chebyshev case}

Chebyshev polynomials of second kind $T_n(x)$ are symmetric Jacobi polynomials 
with $\alpha = \beta = -1/2$ and a different normalization: 
\begin{equation}\label{Cheb Jac}
P_n^{(-1/2,-1/2)}(x) = \frac{\left(1/2\right)_n}{n!} 
T_n(x), \ n\geqslant 0.
\end{equation} 

\begin{corollary}\label{cor Cheb any j}  Let $m$ and $n$ be two positive 
integers with $n\geqslant m$. 
The coefficients $\widetilde{\rho}_{k,n}^m$ in the expansion 
\begin{equation*}\label{conv Cheb}
\int_{-1}^{x+1} T_m(x-t)T_n(t) \mathrm{d}t = \sum_{k=0}^{m+n+1} 
\widetilde{\rho}_{k,n}^{m} T_{k}(x+1), 
\end{equation*} 
can be expressed as 
\begin{subequations}
\begin{empheq}{alignat=3}
\label{Cheb Rjn j geq m}
\widetilde{\rho}_{j,n}^{m} &= \sum_{\nu=\max(1, |j-n|)}^{m+1}
\widetilde{\varpi}_{j,\nu}^{m,n}  & \quad \text{ for } j\geqslant 
\max(m+1,n-m-1), \\
\label{Cheb Rjn j geq m2}
\widetilde{\rho}_{j,n}^{m} &=0 & \quad \text{ for } m+1 \leqslant j \leqslant 
n-m-2, \text{ if } n \geqslant 2m+3, \\
\label{Cheb Rjn j leq m} 
\widetilde{\rho}_{j,n}^{m}&= \sum_{\nu=1}^{j} \widetilde{\varpi}_{j,\nu}^{n,m} 
+ \sum_{\nu=j+1}^{n+1} \widetilde{d}_{\nu,j,n}^{m} & \quad \text{ for } 0 
\leqslant j \leqslant m,
\end{empheq}
where
\begin{equation}\label{Tche inner coef}
\widetilde{\varpi}_{j,\nu}^{m,n} = \frac{(-1)^{m}\, 2^{1-2\nu} n (-m)_{\nu-1} 
(m)_{\nu-1} (-\nu)_{\frac{n+\nu-j}{2}} \left(\frac{n+\nu-j+2}{2}\right)_{j-\nu -1
}}{\left(1/2\right)_{\nu-1} \left(\frac{n+\nu+j}{2} \right)!}
\end{equation}
for even $(n+\nu-j)$ and $\widetilde{\varpi}_{j,\nu}^{m,n} = 0$ otherwise. Here,
\begin{equation}\label{Tche d}
\widetilde{d}_{\nu,j,n}^{\ m} = \frac{2^{2-\delta _{0,j}} (-1)^{m+n} (-n)_{\nu 
-1} (n)_{\nu-1}\left(\nu-1/2\right)}{(\nu-j)!(j+\nu)!}
\pFq{4}{3}{1,-m,m,\nu +1/2}{1/2,-j+\nu+1,j+\nu+1}{1}.
\end{equation}
\end{subequations}
\end{corollary}
\begin{proof} 
Equation \eqref{Cheb Jac} suggests 
\begin{equation}\label{Rcheb to RJacsym}
\widetilde{\rho}_{j,n}^m = \frac{m! n!(1/2)_j}{(1/2)_m(1/2)_n j!}
\rho_{j,n}^{m;(-1/2,-1/2)},
\end{equation}
which leads to \eqref{Cheb Rjn j geq m}-\eqref{Cheb Rjn j leq m} with 
\begin{equation*}
\widetilde{\varpi}_{j,\nu}^{m,n} = \frac{m! n! (1/2)_j} {(1/2)_m(1/2)_n j!} 
\varpi_{j,\nu}^{m,n}(-1/2,-1/2).
\end{equation*}
Setting $\alpha = -1/2$ in \eqref{varpi sym Jac even} and simplifying give 
\eqref{Tche inner coef}.

Similarly, \eqref{Rcheb to RJacsym} implies
\begin{equation}\label{cheb d}
\widetilde{d}_{\nu,j,n}^{\ m} = \frac{m! n! (1/2)_j}{(1/2)_m(1/2)_n j!} 
d_{\nu,j,n}^{m}(-1/2, -1/2), 
\end{equation}
where
\begin{equation}\label{cheb d j geq 1}
\begin{multlined}
d_{\nu,j,n}^{m}(-1/2, -1/2) = \frac{4(-1)^{m+n+1+\nu} (-1/2 + \nu)(j+1/2)_{m-j} 
(1/2)_n (n)_{\nu-1}}{m!(n+1-\nu)!(\nu-j)!(j+1)_{\nu}}\\[2mm]
\times \pFq{4}{3}{1,-m, 1/2 +\nu ,m}{\nu-j+1, 1/2,j+\nu +1}{1}
\end{multlined}
\end{equation}
for $j\geqslant 1$. When $j=0$, \eqref{JacSym d} gives
\begin{align}
d_{\nu,0,n}^{m}(-1/2,-1/2) =& \lim_{\alpha\to -1/2} d_{\nu,0,n}^{m}(\alpha,
\alpha) \nonumber \\
=& \frac{2(-1)^{m+n+1+\nu} (\nu-1/2)(1/2)_{m} (1/2)_n (n)_{\nu-1}}{m!(n+1-\nu)! 
(\nu!)^2}\pFq{4}{3}{1,-m, 1/2 +\nu ,m }{\nu+1, 1/2 ,\nu +1}{1} \label{cheb d j 
eq 0}.
\end{align}
By combining \eqref{cheb d j geq 1} and \eqref{cheb d j eq 0}, we have
\[
\begin{multlined}
d_{\nu,j,n}^{m}(-1/2, -1/2) = \frac{2^{2-\delta_{0,j} }(-1)^{m+n+1+\nu} (-1/2 
+\nu)(j+1/2)_{m-j} (1/2)_n (n)_{\nu-1}}{m!(n+1-\nu)!(\nu-j)!(j+1)_{\nu}}\\[2mm]
\times \pFq{4}{3}{1,-m, 1/2 +\nu ,m}{\nu-j+1, 1/2,j+\nu +1}{1},
\end{multlined}
\]
which, along with \eqref{cheb d}, yields \eqref{Tche d}.
\end{proof}

\begin{figure}[ht]
\centering
\includegraphics[scale=0.6]{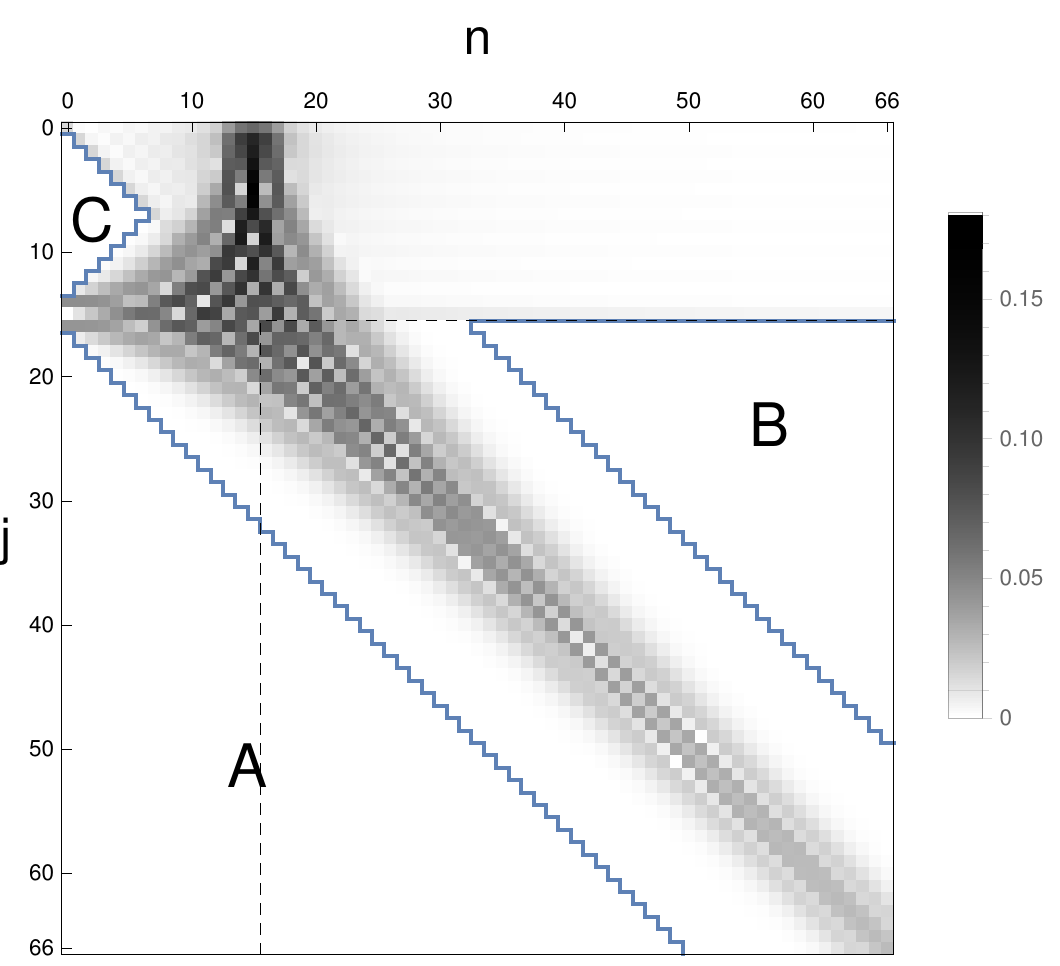}
\caption{The magnitude plot of the Chebyshev coefficients 
$\widetilde{\rho}_{j,n}^{15}$.}\label{FIG:cheb}
\end{figure}

Figure \ref{FIG:cheb} shows the magnitudes of the coefficients $\rho_{j,n}^{m}$ 
for $m = 15$ and the region A, B, and C where $\rho_{j,n}^{m}$ are exactly 
zero. The coefficients that satisfy the symmetry are encompassed by the dashed 
lines. 

\section{Laguerre case}\label{sec: Lag case}
Laguerre polynomials parameterized by $\alpha$ with $\Re(\alpha)>-1$ 
can be defined in terms of terminating hypergeometric functions with the 
following commonly-used normalization
\begin{equation*}\label{Lag 1F1}
L_n^{(\alpha)}(x) = \frac{(1+\alpha)_n}{n!} \pFq{1}{1}{-n}{\alpha+1}{x},~~~~~~ 
n\geqslant 0,
\end{equation*}
and they satisfy the orthogonality relations 
\begin{equation*}
\int_{0}^{+\infty} x^{\alpha} \text{e}^{-x} L_k^{(\alpha)}(x) L_n^{(\alpha)}(x) 
\md x = \frac{\Gamma(\alpha+n+1)}{n!} \delta_{n,k}.
\end{equation*}

Laguerre polynomials with different parameters are linearly related via 
\cite[(3.46)]{AsFi}
\begin{equation}\label{Ln alpha sum Ln beta}
L_{n}^{(\alpha)}(x) = \sum_{k=0}^n \frac{(\alpha-\beta)_{n-k}}{(n-k)!} 
L_{k}^{(\beta)}(x),
\end{equation}
and the $p$-th derivatives of $L_{n+p}^{(\alpha)}$ equals to $L_n^{(\alpha+p)}$ 
up to a sign:
\begin{equation}\label{Ln alpha deriv}
\frac{\md^p}{\md x^p}L_{n+p}^{(\alpha)}(x) = (-1)^p L_{n}^{(\alpha+p)}(x).
\end{equation}
We combine \eqref{Ln alpha sum Ln beta} and \eqref{Ln alpha deriv} to have 
\begin{equation}\label{cc DpLag}
\frac{\md^p}{\md x^p}L_{n+p}^{(\alpha)}(x) = 
\sum_{k=0}^n\gamma_{n,k}^{(p,q)}(\alpha;L)\frac{\md^q}{\md 
x^q}L_{k+q}^{(\alpha)}(x),~~~\text{ where }~~~\gamma_{n,k}^{(p,q)}(\alpha;L) 
= \frac{(-1)^{p+q} (p-q)_{n-k}}{(n-k)!}.
\end{equation}

The Laguerre representation of monomials can be found in, for example, 
\cite[p.207]{Rainville}:
\begin{equation}\label{xn to Ln}
x^n = \sum_{k=0}^n b_{n,k}(\alpha) L_k^{(\alpha)}(x), ~~~\text{ where }~~~
b_{n,k}{(\alpha)} = \frac{(-n)_k \Gamma (n+\alpha +1)}{\Gamma (k+\alpha +1)}
= (-1 )^k n! \dL{k}{n}.
\end{equation}

It also follows from \eqref{Ln alpha deriv} the value of $L_n^{(\alpha)}(x)$ or 
that of its derivatives at $x=0$:
\begin{equation}\label{dpL at minus1}
\dL{p}{n} = (-1)^p \frac{(1+\alpha+p)_{n-p}}{(n-p)!}.
\end{equation}

The Chu-Vandermonde's identity used repeatedly in the rest of this section:
\begin{equation}\label{Lag prod}
\frac{(\alpha+\beta+2)_n}{n!} = \sum_{k=0}^n 
\frac{(\alpha+1)_k}{k!}\frac{(\beta+1)_{n-k}}{(n-k)!}, 
\end{equation}
is valid for any complex numbers $\alpha$ and $\beta$. Observe that for 
$\alpha,\beta>-1$, it corresponds to \eqref{Ln alpha sum Ln beta} evaluated at 
$x=0$. 

For the case of $\alpha=0$, convolution of Laguerre polynomials $L_n^{(0)}(x)$ 
and $L_m^{(0)}(x)$ is sparse in the sense that only two coefficients of its 
$L_n^{(0)}$-series representation are nonzero:
\begin{equation}\label{conv Lag 0}
\int_{0}^{x} L_m^{(0)}(x-t)  L_n^{(0)}(t) \mathrm{d}t = -L_{m+n+1}^{(0)}(x) + 
L_{m+n}^{(0)}(x).
\end{equation}
This well-known result can be found in, for example, \cite[Eq. (18.17.2)]{DLMF} 
or \cite[Eq. (7.411.4)]{GR}. There does not seem to be any similar formula 
addressing the cases with other values of $\alpha$ and this is what we present 
in Theorem \ref{thm Lag any j v2} below.

\begin{theorem}\label{thm Lag any j v2}
Let $m$ and $n$ be two positive integers with $n\geqslant m$ and $\alpha$ be a 
complex number with $\Re(\alpha)>-1$. When $n\geqslant m+1$, the 
$\widehat{\rho}$-coefficients in the expansion
\begin{equation*}\label{gen conv Lag}
\int_{0}^{x} L_m^{(\alpha)}(x-t)  L_n^{(\alpha)}(t) \mathrm{d}t = 
\sum_{j=0}^{m+n+1} \widehat{\rho}_{j,n}^{m;(\alpha)} L_j^{(\alpha)}(x)
\end{equation*}
are given by
\begin{subequations}\label{n geq m+1}
\begin{empheq}[left={ \widehat{\rho}_{j,n}^{m;(\alpha)}=\empheqlbrace\,}]{alignat=2}
& -\frac{(\alpha-1)_{m+n+1-j}}{ (m+n+1-j)!} \hspace{2cm}& \text{for } 
n+1\leqslant j\leqslant m+n+1, \label{Lag Rjn j geq n+1} \\
& \frac{(\alpha)_m}{m!} &\text{for } j=n, \label{Lag Rjn j eq n}\\
& \ 0 & \text{for } m+1\leqslant j\leqslant n-1 \hspace{5mm}\mbox{ (when } n 
\geqslant m+2 \mbox{)}, \label{Lag Rjn m+1 leq j leq n-1}\\
& \frac{(\alpha)_{n+1}}{(n+1)!} & \text{for }j=m, \label{Lag Rjn j eq m}\\
& \frac{(\alpha-1)_{m+n+1-j}}{(m+n+1-j)!} & \text{for } 0\leqslant j\leqslant 
m-1. \label{Lag Rjn j leq m}
\end{empheq}
\end{subequations}
In case of $n = m$, the $\widehat{\rho}$-coefficients become
\begin{subequations}\label{n eq m}
\begin{empheq}[left={ 
\widehat{\rho}_{j,n}^{m;(\alpha)}=\empheqlbrace\,}]{alignat=2}
& -\frac{(\alpha-1)_{2m+1-j}}{(2m+1-j)!} \hspace{3cm}& \text{for } m+1\leqslant 
j\leqslant 2m+1, \label{Lag Rjn j geq n+1 2} \\
& \frac{(\alpha)_{m+1}}{(m+1)!} + \frac{(\alpha)_m}{m!} & \text{for } j=m, 
\label{Lag Rjn j eq n 2}\\
& \frac{(\alpha-1)_{2m+1-j}}{(2m+1-j)!} & \text{for }0\leqslant j\leqslant m-1. 
\label{Lag Rjn j leq m 2}
\end{empheq}
\end{subequations}
\end{theorem}

\begin{proof}
With $a=0$, Theorem \ref{thm: rho k geq m} gives
\begin{subequations}
\begin{equation}\label{proofRLag kn small v2}
\widehat{\rho}_{j,n}^{m;(\alpha)} = \hspace{-5mm} \sum_{\nu=\max(1,j-n)}^{m+1} 
\hspace{-3mm} \gamma_{n-j+\nu,0}^ {(j-\nu,j)} \dL{\nu-1}{m} ~~~~
\text{ for } j\geqslant m+1
\end{equation}
and
\begin{equation}\label{proofRLag kn small v2 2}
\begin{multlined}
\widehat{\rho}_{j,n}^{m;(\alpha)} = \sum\limits_{\nu=1}^{j} 
\gamma_{m-j+\nu,0}^{(j-\nu,j)} \dL{\nu-1}{n} \\
\hspace{1cm}+ \sum\limits_{\nu=j+1}^{n+1} \left( \dL{\nu-1}{n} 
\sum\limits_{p=0}^{m} \frac{b_{p+\nu,j}}{(p+\nu)!} \dL{p}{m}\right) ~~~~\text{ 
for } 0\leqslant j\leqslant m,
\end{multlined}
\end{equation}
\end{subequations}
where the $\gamma$-coefficients, the $b$-coefficients, and the derivatives of 
Laguerre polynomials at $x=0$ are given by \eqref{cc DpLag}, \eqref{xn to Ln}, 
and \eqref{dpL at minus1}, respectively. The main task now boils down to the 
simplification of \eqref{proofRLag kn small v2} and \eqref{proofRLag kn small v2 
2}. Our discussion branches for different ranges of $j$.

\underline{\it For $j \geqslant m+1$:}

Equation \eqref{proofRLag kn small v2} gives
\begin{equation}\label{pf Lag rho big}
\widehat{\rho}_{j,n}^{m;(\alpha)} = \sum_{\nu=\max(1,j-n)}^{m+1} 
\gamma_{n-j+\nu,0}^{(j-\nu,j)} \dL{\nu-1}{m} = -\hspace{-4mm} 
\sum_{\nu=\max(1,j-n)}^{m+1} \hspace{-2mm} 
\frac{(-\nu)_{-j+n+\nu}(\alpha+\nu)_{m-\nu+1}}{(n-j+\nu)!(m-\nu +1)!}.
\end{equation}

For $n+1\leqslant j\leqslant m+n+1$, the summation index $\nu$ in \eqref{pf 
Lag rho big} runs from $j-n$. Making a change of variable $\nu \to \nu+j-n$ 
and using the fact that $(-j+n-\nu)_{\nu} = (-1)^\nu (j-n+1)_{\nu}$ and 
$(\alpha+j-n+\nu)_{m+n+1-j-\nu}=(-1)^{m+n+1-j-\nu}(-\alpha-m)_{m+n+1-j-\nu}$ 
leads to
\[
\widehat{\rho}_{j,n}^{m;(\alpha)} = (-1)^{m+n+j}\sum_{\nu=0}^{m+n+1-j} 
\frac{(j-n+1)_{\nu} (-\alpha-m)_{m+n+1-j-\nu}}{(m+n+1-j-\nu)!\nu!},
\]
from which we obtain \eqref{Lag Rjn j geq n+1} by applying \eqref{Lag prod} and 
noting $(-\alpha-n-m+j+1)_{m+n+1-j} = (-1)^{m+n-j+1}(\alpha-1)_{m+n+1-j}$. In 
case of $m=n$, this gives \eqref{Lag Rjn j geq n+1 2}.

For $j=n$, by noting that $(-\nu)_\nu = (-1)^{\nu} \nu!$ we simplify \eqref{pf 
Lag rho big} to find
\[
\widehat{\rho}_{n,n}^{m;(\alpha)} = -\sum_{\nu=1}^{m+1} \frac{(-1)^{\nu}(\alpha 
+\nu)_{m-\nu +1}}{(m-\nu+1)!} = (-1)^m \sum_{\nu=1}^{m+1} \frac{(-\alpha-m 
)_{m-\nu+1}}{(m-\nu+1)!} = (-1)^m \sum_{\nu=0}^{m} \frac{(-\alpha-m)_{\nu 
}}{\nu!}= \frac{(\alpha)_m}{m!},
\]
where the first equality follows the sign-flip trick $(\alpha +\nu)_{m-\nu+1} = 
(-1)^{m-\nu+1} (-\alpha-m)_{m-\nu+1}$ and the second is due to a change of 
variable with $m+1-\nu$ in place of $\nu$. The last equality is obtained by 
using \eqref{Lag prod}. This proves \eqref{Lag Rjn j eq n}.

If $n\geqslant m+2$, it is possible that $m+1 \leqslant j \leqslant n-1$. In 
this case, $(-\nu)_{-j+n+\nu}=0$, which leads to \eqref{Lag Rjn m+1 leq j leq 
n-1}.

\underline{\it For $0\leqslant j \leqslant m$:}

We denote the two $\nu$-sums in \eqref{proofRLag kn small v2 2} by $S_1$ and 
$S_2$, respectively. The first sum
\[
S_1 = \sum_{\nu=1}^j \gamma_{m-j+\nu,0}^{(j-\nu,j)} \dL{\nu-1}{n} = 
-\sum_{\nu=1}^j \frac{(-\nu)_{m-j+\nu}(\alpha+\nu)_{n-\nu+1}}{(m-j+\nu)!(n-\nu 
+1)!}
\]
vanishes for any $j\leqslant m-1$, since $(-\nu)_{\tau}=0$ for $\tau \geqslant 
\nu+1$. Therefore, we only have to consider the case of $j=m$, for which
\begin{equation}\label{S1}
\begin{multlined}
S_1 = \sum_{\nu=1}^{m}\frac{(-1)^{\nu+1}(\alpha+\nu)_{n-\nu+1}}{(n+1-\nu)!}
= (-1)^n \hspace{-4mm} \sum_{\nu=n-m+1}^{n} \hspace{-4mm}
\frac{(-\alpha-n)_{\nu}}{\nu!}\\[2mm]
\hspace{3cm}=(-1)^n\left(\sum_{\nu=0}^{n}-\sum_{\nu=0}^{n-m}\right)\frac{
(-\alpha-n)_ { \nu } } { \nu!} = 
\frac{(\alpha)_n}{n!}-\frac{(-1)^{m}(\alpha+m)_{n-m}}{(n-m)!},
\end{multlined}
\end{equation}
where the first equality follows from a change of variable and the last is 
obtained by applying \eqref{Lag prod} to each of the two sums.

The second sum in \eqref{proofRLag kn small v2 2} reads
\begin{align*}
S_2 &= \sum\limits_{\nu=j+1}^{n+1} \left(\dL{\nu-1}{n} \sum\limits_{p=0}^{m} 
\frac{b_{p+\nu,j}}{(p+\nu)!} \dL{p}{m}\right)\\
&=\sum\limits_{\nu=j+1}^{n+1} \left(\frac{(-1)^{\nu-1} (\alpha+\nu)_{n-\nu 
+1}}{(n-\nu+1)!} \sum\limits_{p=0}^{m} \frac{(-1)^p (-p-\nu)_j 
\Gamma(p+\alpha+\nu+1) (p+\alpha+1)_{m-p}}{(p+\nu)! 
\Gamma(j+\alpha+1)(m-p)!}\right)\\
&=(-1)^{m+n+j} \sum\limits_{\nu=j+1}^{n+1} \left( 
\frac{(-\alpha-n)_{n-\nu+1}}{(n-\nu+1)!} \sum\limits_{p=0}^{m} 
\frac{(\alpha+j+1)_{p+\nu-j} (-\alpha-m)_{m-p}}{(p+\nu-j)!(m-p)!} \right),
\end{align*}
where, to obtain the last equality, we have used the identities $(\alpha 
+\nu)_{n-\nu +1}=(-1)^{n-\nu+1} (-\alpha-n)_{n-\nu +1}$ and $(p+\alpha+1)_{m-p}= 
(-1)^{m-p} (-\alpha-m)_{m-p}$. Making the changes of variables $\nu\to n-\nu+1$ 
and $p\to m-p$, we have
\begin{align*}
S_2 &= (-1)^{m+n+j} \sum\limits_{\nu=0}^{n-j}\left(\frac{(-\alpha-n)_{\nu 
}}{\nu!} \sum\limits_{p=0}^{m}\frac{(j+\alpha+1)_{m+n+1-j-\nu-p}(-\alpha-m)_{p}}
{(m+n+1-j-\nu-p)!p!} \right)\\
&= (-1)^{m+n+j} \sum\limits_{\nu=0}^{n-j} \Bigg[\frac{(-\alpha-n)_{\nu}}{\nu!} 
\left(\sum\limits_{p=0}^{m+n+1-j-\nu}-\sum\limits_{p=m+1}^{m+n+1-j-\nu}
\right) \frac{(j+\alpha+1)_{m+n+1-j-\nu-p}(-\alpha-m)_{p}}{(m+n+1-j-\nu-p)!p!}
\Bigg].
\end{align*}
By virtue of \eqref{Lag prod}, the first $p$-summation in the last equation 
simplifies to $\ds \frac{(j-m+1)_{m+n+1-j-\nu}}{(m+n+1-j-\nu)!}$, which vanishes 
for $0\leqslant j\leqslant m-1$ and equals $1$ for $j=m$. Now we change the 
variable $p\to p+m+1$ in the second $p$-summation to obtain
\begin{align*}
S_2 &= (-1)^{m+n+j} \sum\limits_{\nu=0}^{n-j} \Bigg[ 
\frac{(-\alpha-n)_{\nu}}{\nu!} \left(\delta_{j,m} - 
\sum\limits_{p=0}^{n-j-\nu}\frac{(j+\alpha +1)_{n-j-p-\nu}(-\alpha-m)_{m+1+p}}
{(n-\nu-j-p)!(m+1+p)!} \right) \Bigg] \\
&= (-1)^{m+n+j} \Bigg[\sum\limits_{\nu=0}^{n-j} \frac{(-\alpha-n)_{\nu}}{\nu!} 
\delta_{j,m}-\sum\limits_{p=0}^{n-j} \left(\frac{(-\alpha-m)_{m+1+p}}{(m+1+p)!}
\sum\limits_{\nu=0}^{n-j-p} \frac{(j+\alpha +1)_{n-j-p-\nu} 
(-\alpha-n)_{\nu}}{(n-\nu-j-p)!\nu!}\right) \Bigg],
\end{align*}
where we have swapped the order of $\nu$- and $p$-sums. Now, both the $\nu$-sums 
can be simplified using \eqref{Lag prod}:
\begin{align}
S_2 &= (-1)^{m+n+j} \left[\frac{(-\alpha-n+1)_{n-j}}{(n-j)!}\delta_{j,m}
-\sum\limits_{p=0}^{n-j} \frac{(j-n+1)_{n-j-p}(-\alpha-m)_{m+1+p}} 
{(n-j-p)!(m+1+p)!} \right] \nonumber \\
&= (-1)^{m+n+j} \left[(-1)^{n-j}\frac{(\alpha+j)_{n-j}}{(n-j)!}\delta_{j,m}
- \sum\limits_{p=0}^{n-j} \frac{(j-n+1)_{p}(-\alpha-m)_{m+1+n-j-p}} 
{p!(m+1+n-j-p)!} \right], \label{S2 2}
\end{align}
where we have used the sign-flip trick for the first term in the brackets and 
changed variable $p\to n-j-p$ in the $p$-sum. 

When $j=m=n$, only the zeroth summand is left for the $p$-sum in \eqref{S2 2}:
\begin{equation*}\label{S2 n n n}
S_2 = (-1)^{m} \left[1 - \frac{(-\alpha-m)_{m+1}}{(m+1)!} \right]
= (-1)^{m} + \frac{(\alpha)_{m+1}}{(m+1)!},
\end{equation*}
which, combined with \eqref{S1} for $m=n$, gives \eqref{Lag Rjn j eq n 2}. For 
all other cases, the upper limit of the $p$-sum in \eqref{S2 2} can be bumped 
to $m+n+1-j$ as $(j-n+1)_{p}$ vanishes for $p\geqslant n-j+1$. Simplifying the 
sum by \eqref{Lag prod}, followed by using the sign-flip trick again, we 
finally obtain
\begin{equation*}\label{S2}
S_2 = \frac{(-1)^m (\alpha+m)_{n-m}}{(n-m)!} \delta_{j,m} +
\frac{(\alpha-1)_{m+n+1-j}}{(m+n+1-j)!},
\end{equation*}
which, together with \eqref{S1}, yields \eqref{Lag Rjn j eq m}, \eqref{Lag Rjn 
j leq m}, and \eqref{Lag Rjn j leq m 2}.
\end{proof}

\begin{figure}[b]
\centering
\begin{subfigure}[b]{0.32\textwidth}
\includegraphics[scale=0.380]{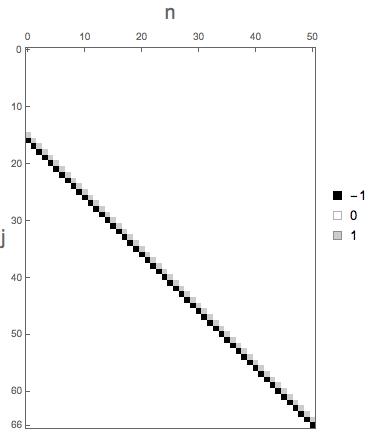}\caption{$\alpha = 0$} 
\label{FIG:LagConv_alpha0}
\end{subfigure}
\begin{subfigure}[b]{0.32\textwidth}
\includegraphics[scale= 0.380]{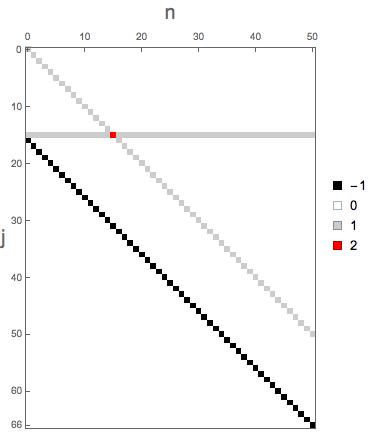}\caption{$\alpha = 1$} 
\label{FIG:LagConv_alpha1}
\end{subfigure}
\begin{subfigure}[b]{0.32\textwidth}
\includegraphics[scale= 0.38]{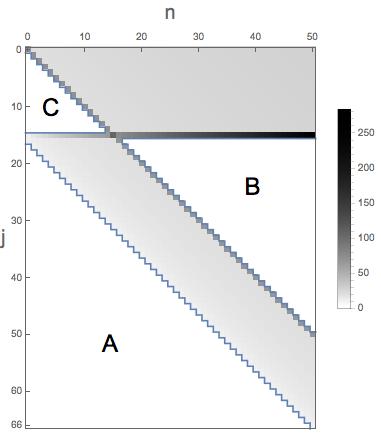}\caption{$\alpha = 
2.5$}\label{FIG:LagConv_alpha2dot5}
\end{subfigure}
\caption{The magnitude plot of the Laguerre coefficients 
$\widehat{\rho}_{j,n}^{15;(\alpha)}$ for $0 \leqslant n \leqslant 
50$.}\label{FIG:lag}
\end{figure}

\begin{remark}\label{rem:lag_alpha0}
When $\alpha=0$, $\widehat{\rho}_{j,n}^{m;(0)}$ given by \eqref{n geq m+1} 
becomes zero for any $0 \leqslant j \leqslant m+n-1$, while 
$\widehat{\rho}_{m+n+1,n}^{m;(0)} = -1$ and $\widehat{\rho}_{m+n,n}^{m;(0)} 
= 1$, which recovers \eqref{conv Lag 0}. Same for $\widehat{\rho}_{j,m}^{m;(0)}$ 
given by \eqref{n eq m}.
\end{remark}

\begin{remark}\label{rem:lag_alpha1}
When $\alpha=1$, $\widehat{\rho}$-coefficients also enjoy sparsity, suggested by 
Theorem \ref{thm Lag any j v2}:
\[
\int_{0}^{x} L_m^{(1)}(x-t) L_n^{(1)}(t) \mathrm{d}t = -L_{m+n+1}^{(1)}(x)
+ L_{n}^{(1)}(x) + L_{m}^{(1)}(x).
\]
\end{remark}

Figure \ref{FIG:lag} shows the magnitudes of the coefficients 
$\widehat{\rho}_{j,n}^{m;(\alpha)}$ with $m = 15$ for (a) $\alpha = 0$, (b) 
$\alpha = 1$, and (c) $\alpha = 2.5$. The $n$-th column in this matrix plot 
corresponds to $\rho_{j,n}^{15;(\alpha)}$. Panes (a) and (b) confirm Remarks 
\ref{rem:lag_alpha0} and \ref{rem:lag_alpha1}, respectively. With $\alpha = 
2.5$, the plot in pane (c) is representative for a general case where three 
regions of exact zeros are indicated by solid lines. The exact zeros in Region 
A is again due to the fact that the convolution of $L_{15}^{(2.5)}$ and 
$L_n^{(2.5)}$ is a polynomial of degree $n+16$, while the zeros in Region B  
corresponds to \eqref{Lag Rjn m+1 leq j leq n-1}. The zeros in region C are also 
due to \eqref{Lag Rjn m+1 leq j leq n-1} but with the roles of n and m 
exchanged.



\section{Closing remarks}
In this paper, we have derived the explicit formulas for the coefficients in 
the series representation for the convolution of the elements in a polynomial 
sequence. Particularly, the results are significantly simplified when the 
polynomial sequence is formed by classical orthogonal polynomials of Jacobi- or 
Laguerre families. 

As mentioned in Section 1, a most common scenario seen in practice is that the 
functions to be convolved are approximated by polynomial series
\begin{equation}\label{fM}
f(x) \approx f_M(x) = \sum_{m=0}^M a_m P_m(x) \quad \text{ and }\quad g(x) 
\approx g_N(x) = \sum_{n=0}^N b_n P_n(x),
\end{equation}
where the coefficients $a_m$ and $b_n$ are known or computationally obtainable. 
The convolution of $f_M(x)$ and $g_N(x)$ can be represented by a third series 
\begin{equation*}
h_{M+N+1}(x) = \sum_{k=0}^{M+N+1} c_k P_k(x)
\end{equation*}
with the coefficient vector $\underline{c}$ being the product of the convolution 
matrix $R$ and the coefficient vector $\underline{b}$, given by \eqref{cRb}. In 
fact, \eqref{gen conv} and \eqref{fM} gives
\begin{equation*}\label{int Tn f expanded}
\int_{-a}^{x+a} f_M(x-t) P_n(t) \dt = \sum_{m=0}^{M} a_m \sum_{j=0}^{m+n+1} 
\rho_{j,n}^{m} P_j(x+a) \quad \text{ for } 0 \leqslant n \leqslant N,
\end{equation*}
which, with the order of the sums swapped, becomes 
\begin{equation}\label{conv fPn 1}
\int_{-a}^{x+a} f_M(x-t)P_n(t) \dt = \sum_{j=0}^{n}\sum_{m=0}^M a_m 
\rho_{j,n}^{m} P_j(x+a) + \sum_{j=n+1}^{M+n+1}\sum_{m=j-(n+1)}^M a_m 
\rho_{j,n}^{m} P_j(x+a).
\end{equation}
Noting that
\begin{equation}\label{conv fPn 2}
\int_{-a}^{x+a} f_M(x-t)P_n(t) \dt = \sum_{j=0}^{M+n+1} R_{j,n} P_j(x+a),   
\end{equation}
where $R_{j,n}$ is the $(j,n)$-entry of $R$. Matching \eqref{conv fPn 1} and 
\eqref{conv fPn 2}, we have 
\begin{equation}\label{conv fPn 3}
R_{j,n} = \sum_{m=\max(0,n+1-j)}^M a_m \rho_{j,n}^{m}. 
\end{equation}

As seen from the magnitude plots in Section 3, many of the $\rho$-coefficients, 
though they are non-zero in an exact sense, can be deemed as zeros in floating 
point arithmetic due to their tiny magnitudes. An exciting extension of the 
results in this paper is the investigation of the asymptotic behavior of the 
$\rho$-coefficients for large $j$ and $n$ using the newly derived explicit 
formulas. The asymptotics will help us identify via \eqref{conv fPn 3} the 
entries of $R$ that can be safely zeroed in numerical computation and, 
therefore, enable a faster construction of the convolution matrix $R$.

On the other hand, the explicit formulas for the convolution coefficients may 
also reveal the numerical rank of the convolution matrix $R$. The low rank 
property of $R$ or its subparts, if any, could lead to potential speed-up in 
either construction of $R$ or fast algorithms for convolution. We save these 
possibilities for a future work.

\thanks{\textbf{Acknowledgment} We would like to thank Erik Koelink for sharing 
his insights on the convolution product of two Jacobi polynomials and Alfredo 
Dea\~no for enlightening discussions. }

\bibliographystyle{amsplain}
\bibliography{ConvClassical}

\providecommand{\bysame}{\leavevmode\hbox to3em{\hrulefill}\thinspace}
\providecommand{\MR}{\relax\ifhmode\unskip\space\fi MR }
\providecommand{\MRhref}[2]{%
  \href{http://www.ams.org/mathscinet-getitem?mr=#1}{#2}
}
\providecommand{\href}[2]{#2}
\begin{thebibliography}{10}

\bibitem{AAR}
George~E Andrews, Richard Askey, and Ranjan Roy, \emph{Special functions},
  Cambridge University Press, Cambridge, 1999.

\bibitem{AsFi}
Richard Askey and James Fitch, \emph{Integral representations for jacobi
  polynomials and some applications}, J. Math. Anal. Appl. \textbf{26} (1969),
  no.~2, 411--437.

\bibitem{GR}
Izrail~Solomonovich Gradshteyn and Iosif~Moiseevich Ryzhik, \emph{Table of
  integrals, series, and products}, Academic press, 2014.

\bibitem{Ismail}
Mourad E~H Ismail, \emph{Classical and quantum orthogonal polynomials in one
  variable}, Encyclopedia of Mathematics and its Applications, vol.~98,
  Cambridge University Press, Cambridge, 2009.

\bibitem{lin}
Peter Linz, \emph{Analytical and numerical methods for volterra equations},
  SIAM, 1985.

\bibitem{Mar}
Pascal Maroni, \emph{Semi-classical character and finite-type relations between
  polynomial sequences}, Appl. Numer. Math. \textbf{31} (1999), no.~3,
  295--330.

\bibitem{RochaMar}
Pascal Maroni and Z{\'e}lia da~Rocha, \emph{Connection coefficients between
  orthogonal polynomials and the canonical sequence: an approach based on
  symbolic computation}, Numer. Algorithms \textbf{47} (2008), no.~3, 291--314.

\bibitem{DLMF}
Frank W~J Olver, Daniel~W Lozier, Ronald~F Boisvert, and Charles~W Clark,
  \emph{Nist handbook of mathematical functions}, Cambridge Univeristy Press,
  2010.

\bibitem{Rainville}
Earl~D Rainville, \emph{Special functions}, Chelsea Publishing Co., Bronx,
  N.Y., 1971.

\bibitem{SRM}
Jorge S{\'a}nchez-Ruiz, Pedro~L\'{o}pez Art{\'e}s, Andrei
  Mart{\'\i}nez-Finkelshtein, and Jes\'{u}s Dehesa, \emph{General linearization
  formulae for products of continuous hypergeometric-type polynomials}, J.
  Phys. A-Math. Gen. \textbf{32} (1999), no.~42, 7345.

\bibitem{sze}
G\'{a}bor Szeg\H{o}, \emph{Orthogonal polynomials}, American Mathematical
  Society, 1939.

\bibitem{XL}
Kuan Xu and Ana~F Loureiro, \emph{Spectral approximation of convolution
  operator}, arXiv: 1804.08762.

\end{thebibliography}
\end{document}